\documentclass[a4paper,12pt,parskip=half]{scrartcl}

\usepackage[english]{babel}

\usepackage{amsfonts,amsmath,amsthm,amssymb}

\usepackage[shortlabels]{enumitem}
\usepackage{mathtools}
\usepackage{xfrac}
\usepackage[noadjust]{cite}
\usepackage{lmodern}

\newtheorem{theorem}{Theorem}[section]
\newtheorem{lemma}[theorem]{Lemma}
\newtheorem{corollary}[theorem]{Corollary}
\newtheorem{proposition}[theorem]{Proposition}

\theoremstyle{definition}
\newtheorem{definition}[theorem]{Definition}
\newtheorem{assumption}[theorem]{Assumption}
\newtheorem{remark}[theorem]{Remark}

\theoremstyle{remark}

\numberwithin{equation}{section}

\usepackage{hyperref}

\newcommand{\field}[1]{\mathbb{#1}}
\newcommand{\N}{\field{N}}
\newcommand{\R}{\field{R}}
\newcommand{\C}{\field{C}}

\DeclareMathOperator{\dom}{dom}    
\DeclareMathOperator{\dist}{dist}  
\DeclareMathOperator{\diam}{diam}    
\DeclareMathOperator{\supp}{supp}  

\DeclareMathOperator*{\esssup}{ess\,sup}
\DeclareMathOperator{\inte}{int}
\DeclareMathOperator*{\essinf}{ess\, inf}

\DeclarePairedDelimiter{\abs}{\lvert}{\rvert}
\DeclarePairedDelimiter{\norm}{\lVert}{\rVert}

\renewcommand{\Re}{\operatorname{Re}}
\renewcommand{\Im}{\operatorname{Im}}
\newcommand{\ii}{\imath}
\newcommand{\cA}{\mathcal{A}}
\newcommand{\bA}{\mathbf{A}}
\newcommand{\bX}{\mathbf{X}}
\newcommand{\bT}{\mathbf{T}}

\newcommand{\cJ}{\mathcal{J}}
\newcommand{\cI}{\mathcal{I}}
\newcommand{\cH}{\mathrm{H}}
\newcommand{\cHaus}{\mathcal{H}}
\newcommand{\cL}{\mathcal{L}}
\newcommand{\cV}{\mathcal{V}}
\newcommand{\cW}{\mathcal{W}}
\newcommand{\Hi}{H}
\newcommand{\eps}{\varepsilon}
\newcommand{\dd}{\, \mathrm{d}}
\newcommand{\ft}{\mathfrak{t}}

\newcommand{\cT}{\mathcal{T}}

\newcommand{\ReOp}[1]{\operatorname{\mathfrak{Re}}#1}
\newcommand{\ImOp}[1]{\operatorname{\mathfrak{Im}}#1}

\begin{document}
\author{%
  Hannes Meinlschmidt\footnote{FAU Erlangen-N\"urnberg, Chair for Dynamics, Control, Machine Learning and Numerics (AvH Professorship), Department of Mathematics, Cauerstr.~11, 91058~Erlangen, Germany, and TU Clausthal, Institute of Mathematics, Erzstr.~1, 38678~Clausthal-Zellerfeld, Germany, meinlschmidt@math.fau.de, hannes.meinlschmidt@tu-clausthal.de.
    The work of~H.M.\ was partially funded by the Deutsche Forschungsgemeinschaft (DFG, German Research Foundation) – Projektnummer~551486487.}~~\&~Joachim Rehberg\footnote{WIAS Berlin, Mohrenstr.~39, 10117~Berlin, Germany, rehberg@wias-berlin.de} } \title{Sharp angle estimates for second-order divergence operators} \date{\vspace{-1em}}

\maketitle
\begin{abstract}
  This article is about the (minimal) sector containing the numerical range of the principal part of a linear second-order elliptic differential operator defined by a form on closed subspaces $V$ of the first-order Sobolev space $W^{1,2}(\Omega)$ incorporating mixed boundary conditions.
  We relate the angle of sectoriality and the $\cH^\infty$-angle of realizations of the elliptic operator to two reference angles: the optimal angle of the $L^2$-realization and the optimal angle of the coefficient function.
  We thereby consider the operator in several scales of Banach spaces: the Lebesgue space, the negative Sobolev space, and their interpolation scale.
  For the latter two types of spaces, we rely on recent results regarding the Kato square root property.
  We focus on minimal assumptions on geometry, and we consider both real and complex coefficients, in particular aiming at explicit estimates built on readily accessible problem data.
  New results include an explicit estimate for the angle of sectoriality for the elliptic operator on $L^p(\Omega)$ with complex coefficients without any assumptions on geometry, which reduces to the known sharp angle for real coefficients, a new estimate on the angle of the sector containing the numerical range of a linear, continuous and coercive Hilbert space operator, and an estimate on the $\cH^{\infty}$-angle for the elliptic operator on $L^p(\Omega)$ with real coefficients.
\end{abstract}

\par\vskip\baselineskip\noindent


\section{Introduction}

One of the concepts that mathematicians teach already to undergraduate students in Linear Algebra courses is how to understand the precise behavior of linear transformations by looking at their spectrum.
This paradigm applies equally well to closed operators $B$ on Banach spaces $X$ that arise when studying partial differential equations through the lens of functional analysis.
For example, studying an abstract linear Cauchy problem in $X$ such as
\begin{equation*}
  u'(t) + Bu(t) = f(t) \quad \text{in}~X \qquad (0 < t < T), \qquad u(0) = u_0,
\end{equation*}
one wants to define the operator exponential $S(t) \simeq e^{-tB}$ of $-B$ to determine the solution by the Duhamel formula $u(t) = S(t)u_0 + \int_0^t S(t-s)f(s) \dd s$.
For an unbounded operator $B$, say, a linear second-order elliptic differential operator, the existence of such an operator exponential is tied to its spectral behavior---meaning both the location of the spectrum \emph{and} suitable decay of the resolvent operators outside of it; this is the famous Hille-Yosida theorem.
If the spectrum of $B$ is contained in a sector $\Sigma_\omega$ of opening (half) angle $\omega < \sfrac\pi2$ and the norms of the resolvents $(B-\lambda)^{-1}$ decay like $\sfrac1\lambda$, then the semigroup is \emph{holomorphic}, and the abstract Cauchy problem is \emph{parabolic}~\cite[Section~2.5]{Pazy1983}, with classically parabolic effects such as instant smoothing for the solution $u$.

A disadvantage of the whole idea is that the spectrum of unbounded operators is rather unwieldy.
As a remedy, it is convenient to work with the \emph{numerical range} of $B$,
\begin{equation*}
  N(B) \coloneqq \Bigl\{ u^*(Bu) \colon u \in \dom(B ),~\|u \|_X=1,~u^* \in J(u)\Bigr\}
\end{equation*}
where $J \colon X \rightrightarrows X^*$ is the duality mapping.
If $X^*$ is strictly convex, which is the case for all spaces in which we use the numerical range in this work, then $J$ is single-valued on $X \setminus \{0\}$.
In many interesting cases, the closure of the numerical range contains the spectrum.
The numerical range has found countless applications in the most diverse disciplines of mathematics; we refer to~\cite{GustafsonRao1997,Benzi2021} for a good overview.

In the present work, we concentrate on the case where $B$ corresponds to a linear elliptic second-order differential operator of the form $u \mapsto - \operatorname{div}(\mu\nabla u)$ with a coefficient matrix $\mu$, and we are interested precisely in the \emph{numerical range} of $B$ being contained in a sector $\Sigma_\omega$ with $\omega < \sfrac\pi2$, and in the infimum $\omega(B)$ of all such $\omega$---or, at least, a good estimate for the latter.
In the basic case of a Hilbert space $X$, since the open left half plane belongs to the resolvent set $\rho(B)$ of $B$ due to the Lax-Milgram lemma, the property $N(B) \subseteq \Sigma_\omega$ already implies that $-B$ is the generator of a bounded holomorphic semigroup of contractions of angle $\sfrac\pi2 - \omega$; in particular, $B$ is a \emph{sectorial} operator.
There is also an intimate relation to $B$ admitting a bounded holomorphic functional calculus ($\cH^\infty$-calculus) on the sector $\Sigma_\omega$, which for angles $< \sfrac\pi2$ implies \emph{maximal parabolic regularity} for $B$~\cite{DoreVenni1987}, and to the notion of \emph{spectral sets}~\cite{CrouzeixGreenbaum2019,BickelEtAl2020}; see Section~\ref{sec:defin-prel-results}.

\textbf{Contribution.}
From the foregoing explanations, it becomes apparent that it is most important to understand well the sectors that contain the numerical range of a prototype linear second-order elliptic operator.
Such an operator can be considered in different Banach spaces $X$.
We consider the realization $A$ of the differential operator $u \mapsto - \operatorname{div}(\mu\nabla u)$ with mixed boundary conditions and the uniformly bounded and elliptic coefficient function $\mu \colon \Omega \to \cL(\C^d)$ in $L^2(\Omega)$ for an open set $\Omega$ via the form method~\cite[Chapter~VI]{Kato1980}.
The form will thus be the classical
\begin{equation*}
  \ft \colon W^{1,2}(\Omega) \times W^{1,2}(\Omega) \ni (u,v) \mapsto \int_\Omega (\mu\nabla u, \nabla v) \dd x
\end{equation*}
and we also consider the restriction $\ft_V$ to closed subspaces $V$ of $W^{1,2}(\Omega)$ that carry mixed boundary conditions.
Thus $A$ is the operator corresponding to the form $\ft_V$.
It is also clear that \emph{a priori} it is not at all obvious or expected that there is stability of these sectors when the Banach space $X$ is varied, not even within a natural scale like the $L^p$-spaces.
Thus we are not only concerned with the sectors and their angles \emph{per se}, but also how they depend on certain other reference angles, namely the optimal angle $\omega(\ft_V)$ for the $L^2$-realization of the differential operator and the optimal angle $\omega(\mu)$ for the coefficient function.
Along the way, we collect known results to give a comprehensive picture.
Our new contributions are the following:
\begin{itemize}
\item a new estimate for an angle $\alpha$ such that $N(\ft) \subseteq \Sigma_\alpha$, which improves upon the recent result in~\cite{TerElstLinkeRehberg2022}, uses only explicit data of the coefficient function $\mu$ and transfers to $\ft_V$ (Section~\ref{sec:sharp-incl-numer}),
\item for real coefficients, an interpolation result for the $\cH^\infty$-angle of the realizations $A_p$ of $A$ in $L^p(\Omega)$ (Theorem~\ref{thm:functkalc}),
\item the principle that the deciding sector for $N(A_p)$ is the one that contains a $p$-adapted numerical range of the coefficient function, without any requirement on the geometry of $\Omega$ and its boundary (Theorem~\ref{thm:numerical-range-p-elliptic-operator}),
\item in the complex coefficients case, a new estimate on the angle of a sector containing $N(A_p)$ on the whole $p$-ellipticity interval, which reduces to the sharp angle of Theorem~\ref{thm:chill} for real coefficients (Theorem~\ref{thm:new-complex-angle-estimate}, Corollary~\ref{cor:complex-operator}), and
\item under minimal geometric assumptions, $\omega(\ft_V)$ as a bound for the $\cH^\infty$-angle of the realizations in $W^{-1,q}_D(\Omega)$ and in the interpolation spaces between $L^q(\Omega)$ and $W^{-1,q}_D(\Omega)$ (Section~\ref{sec:oper-negat-sobol}).
\end{itemize}

We proceed as follows.
In Section~\ref{sec:sharp-incl-numer}, we first give the estimate for an angle $\alpha$ such that $N(\ft) \subseteq \Sigma_\alpha$, where $N(\ft)$ is the numerical range of the form $\ft$.
It is sharp in several cases.
Such an estimate is relevant for what was explained above, because $N(A)$ is a dense subset of $N(\ft_V)$, so both numerical ranges share the same optimal angle of accretivity $\omega(\ft_V)$.
It also shows that we can find uniform sectors for operators with different coefficient functions as long as the data entering the estimate is uniform.

We then collect several concepts such as sectoriality, holomorphic semigroups, the $\cH^\infty$-calculus, and related results in Section~\ref{sec:defin-prel-results}, before putting them to use for the $L^2(\Omega)$-operator $A$ in Section~\ref{sec:operator-l2omega}.
Next, we concern ourselves with the realization $A_p$ of $A$ in $L^p(\Omega)$ spaces in Section~\ref{sec:operator-lpomega}.
It turns out that without further assumptions on the geometry of $\Omega$ and its boundary, we will not recover the optimal $L^2$-angle $\omega(\ft_V)$ for $N(A_p)$.
Instead, one obtains a $p$-dependent angle estimate in terms of $\omega(\mu)$, which is known to be sharp (Theorem~\ref{thm:chill}).
In the complex coefficients case, the situation is more dire, and we rely on $p$-ellipticity~\cite{CarbonaroDragicevic2020} and associated results.
Our angle estimate (Corollary~\ref{cor:complex-operator}) reduces to the sharp angle of Theorem~\ref{thm:chill} for real coefficients and to $\omega(\mu) + \arctan(\frac{\abs{p-2}}{2\sqrt{p-1}})$ for symmetric ones.
The latter was already obtained in~\cite{Do2021} under regularity assumptions and is in fact optimal at the level of the coefficients.
A related estimate in terms of uniform data, which does not reproduce the angle of Theorem~\ref{thm:chill} for real coefficients, was obtained in~\cite[Section~8]{BoehnleinEgertRehberg2024}.

Under minimal geometric assumptions established only recently~\cite{BechtelEtAl2026,Bechtel2024a,Bechtel2024b}, it is known that the $L^2$-angle $\omega(\ft_V)$ remains a bound for the $\cH^{\infty}$-angle of $A_p$ whenever the semigroup generated by $-A$ in $L^2(\Omega)$ extrapolates to bounded semigroups on $L^r(\Omega)$ for all $r$ in a neighborhood of $p$.
We describe this in Section~\ref{sec:more-assumpt-geom}.

Finally, in Section~\ref{sec:oper-negat-sobol}, we switch to a weak setting and consider $A$ in $W^{-1,q}_D(\Omega)$, a negative (dual) Sobolev space carrying generalized mixed boundary conditions, for certain values of $q$, and then in interpolation spaces between $L^q(\Omega)$ and $W^{-1,q}_D(\Omega)$.
These spaces have become prominent in the treatment of nonlinear evolution equations with nonsmooth data corresponding to real world problems, such as semiconductor dynamics~\cite{MeinlschmidtRehberg2021,DisserRehberg2019} or the thermistor problem~\cite{MeinlschmidtMeyerRehberg2017a,MeinlschmidtMeyerRehberg2017b}.
Here, we use the operator square root $(A_q+1)^{1/2}$ isomorphism~(\cite{Egert2018,Bechtel2024a}) to move between $L^q(\Omega)$ and $W^{-1,q}_D(\Omega)$.
The geometric assumptions to achieve this build on the ones introduced in Section~\ref{sec:more-assumpt-geom} and already allow to transfer the $L^2$-angle $\omega(\ft_V)$ as a bound for the $\cH^{\infty}$-angle to these spaces.

\paragraph{Notation and conventions.}
The notation used in the article will generally be standard or introduced where needed.
We nevertheless mention a few conventions: All Banach spaces under consideration are complex vector spaces.
Given a Hilbert space $H$, the norm derived from the inner product $(\cdot,\cdot)_H$ is denoted by $\abs{\cdot}_H$.
For the particular Hilbert space $H = \C^d$, we do not label norms and inner products.
Norms on Banach spaces $X$ will be written as usual $\norm{\cdot}_X$.
For operator norms we use the notation $\norm{\cdot}_{X\to Y}$ with the obvious meaning.
We have tried to be consistent with using $z$ for generic complex numbers, $\lambda$ for resolvent points, and $\delta$ or $t$ for real parameters.

\section{A sharp inclusion of the numerical range of the form}\label{sec:sharp-incl-numer}

We start with a few definitions.
Let $0 \leq \theta < \pi$.
Then $\Sigma_\theta \coloneqq \{z \in \C \colon \lvert \arg z\rvert \leq \theta\} \cup \{0\}$ denotes the closed sector of opening angle $2\theta$ in the complex plane.
Let further $H$ be a complex Hilbert space and let $L$ be a linear operator in $H$.
We define its \emph{numerical range} or \emph{field of values} by
\begin{equation*}
  N(L) \coloneqq \Bigl\{(Lx,x)_{\Hi} \colon x \in \dom(L),~\lvert x \rvert_{\Hi} = 1\Bigr\} \subseteq \C.
\end{equation*}
The \emph{numerical radius} of $L$ is defined by $n(L) \coloneqq \sup \{ \lvert \lambda\rvert \colon \lambda \in N(L)\}$.
When $L$ is in fact a bounded operator on $H$, we also employ its decomposition into a sum of a selfadjoint and a skewadjoint operator, using the real and imaginary operator parts of $L$,
\begin{equation*}
  L = \frac{L+L^*}2 + \ii \frac{L-L^*}{2\ii} \eqqcolon \ReOp{L} + \ii \ImOp{L}.%
\end{equation*}
Note that both $\ReOp{L}$ and $\ImOp{L}$ are selfadjoint bounded operators whose numerical ranges are convex subsets of $\R$, and their numerical radii coincide with their respective operator norms.
Moreover, $\Re (Lx,x)_{\Hi} = \bigl((\ReOp{L}) x,x\bigr)_{\Hi}$ and $\Im (Lx,x)_{\Hi} = \bigl((\ImOp{L}) x,x\bigr)_{\Hi}$.

Let us proceed with the following elementary yet quite precise lemma, that we were surprised not to find in the common literature.

\begin{lemma}\label{lem:numrange-newer}
  Let $m > 0 $ and let $L$ be a bounded operator in a Hilbert space $\Hi$ satisfying the coercivity condition $\Re (Lx,x)_{\Hi} \geq m \abs{x}_{\Hi}^2$ for all $x \in \Hi$.
  Then
  \begin{equation*}
    N(L) \subseteq \Sigma_{\alpha} \qquad \text{with} \qquad \tan(\alpha) = \frac{n(\ImOp{L})}m.
  \end{equation*}
\end{lemma}

\begin{proof}
  Let $x \in \Hi$ with $\lvert x \rvert_{\Hi} =1$.
  We have with the coercivity condition
  \begin{equation*}
    \bigl\lvert \Im (Lx,x)_{\Hi} \bigr\rvert = \bigl\lvert \bigl((\ImOp{L})x,x\bigr)_{\Hi} \bigr\rvert \leq n(\ImOp{L}) \leq \frac{n(\ImOp{L})}{m} \Re (Lx,x)_{\Hi}.
  \end{equation*}
  Hence, $\lvert\arg (Lx,x)_{\Hi} \rvert \leq \arctan(\tfrac{n(\ImOp{L})}{m}) = \alpha$ and $N(L) \subseteq \Sigma_\alpha$ follows.
\end{proof}

\begin{remark}
  A variant of Lemma~\ref{lem:numrange-newer} is the main result in~\cite{TerElstLinkeRehberg2022}.
  There, it is shown that
  \begin{equation*}
    N(L) \subseteq \Sigma_{\bar\alpha} \qquad \text{with} \qquad \bar \alpha = \arctan\sqrt{\frac{\lVert L\rVert_{\Hi \to \Hi}^2}{m^2} -1}.
  \end{equation*}
  In fact, this follows from the estimate in Lemma~\ref{lem:numrange-newer}, since
  \begin{equation*}%
    \bigl\lvert \bigl((\ImOp{L})x,x\bigr)_H \bigr\rvert^2 = \bigl\lvert \Im (Lx,x)_H \bigr\rvert^2 \leq
    \bigl\lvert (Lx,x)_H \bigr\rvert^2 - m^2,
  \end{equation*}
  so $n(\ImOp{L})^2 \leq n(L)^2 - m^2 \leq \lVert L\rVert_{\Hi\to\Hi}^2 - m^2$ and thus $\alpha \leq \bar\alpha$.
  However, there are interesting cases where the angle $\alpha$ is indeed substantially smaller than $\bar\alpha$, for example when $L$ is nearly selfadjoint but with large numerical radius.
\end{remark}

\begin{remark}\label{rem:sharpness}
  \begin{enumerate}[(i)]
  \item Note that $\alpha$ is the angle for which the rays in the complex plane starting from $0$ with opening angles $\pm\alpha$ pass through $m\pm\ii n(\ImOp{L})$.
    Since $N(L)$ is contained in $\{w \in \C \colon \Re w \geq m,~\abs{\Im w} \leq n(\ImOp{L})\}$, we see that $\alpha$ is sharp precisely when $z = m+\ii n(\ImOp{L})$ or its complex conjugate $\bar z$ belongs to $\overline{N(L)}$.
    Either point can only be a corner point of the convex set $N(L)$, so if $N(L)$ is closed (in particular, if $H$ is finite-dimensional), $\alpha$ is sharp if and only if $z$ or $\bar z$ is an eigenvalue of $L$~(\cite[Theorem~1]{Donoghue1957}).
    A similar argument was made in the context of Ornstein-Uhlenbeck operators where the angle of holomorphy of the associated semigroup is determined by the spectral angle of a certain matrix, cf.~\cite[Remark~2]{ChillEtAl2005}.
  \item In fact, the possible shapes of the numerical range are rather well understood for \emph{matrices} $L \in \cL(\C^d)$, leading to the following characterization: $\alpha$ is sharp if and only if $L$ is unitarily equivalent to a block matrix of form $\left(\begin{smallmatrix}a & 0 \\ 0 & \cA\end{smallmatrix}\right)$ with $\cA \in \C^{(d-1)\times(d-1)}$ and $a \in \{z,\bar z\}$, see~\cite[Theorem~13]{Kippenhahn1951} for the corner points.
  \end{enumerate}
\end{remark}

From Lemma~\ref{lem:numrange-newer} we see that the numerical range of a coercive bounded operator on a Hilbert space $\Hi$ is contained in a sector $\Sigma_\alpha$ with $\alpha < \sfrac\pi2$.
In particular, such an operator is \emph{accretive}, that is, its numerical range is contained in $\Sigma_{\sfrac\pi2}$.
For an accretive (possibly unbounded) operator $L$, it will be useful to define the \emph{optimal} angle of accretivity
\begin{equation*}
  \omega(L) \coloneq\inf\Bigl\{ \theta \in [0,\sfrac\pi2] \colon N(L) \subseteq \Sigma_\theta\Bigr\}.
\end{equation*}
Since the sectors $\Sigma_\theta$ are closed, the infimum is attained, that is, $N(L) \subseteq \Sigma_{\omega(L)}$.
We next apply Lemma~\ref{lem:numrange-newer} to determine a sector which includes the numerical range of a form.
Therefore, we recall that if $\mathfrak a$ is a sesquilinear form in $\Hi$, then its numerical range is given by
\begin{equation*}
  N(\mathfrak a) \coloneqq \Bigl\{\mathfrak a(x,x) \colon x \in \dom(\mathfrak a),~\abs{x}_H = 1 \Bigr\}.
\end{equation*}
We say that $\mathfrak a$ is an \emph{$\omega$-accretive form} for $\omega \in [0,\sfrac\pi2]$ if $N(\mathfrak a) \subseteq \Sigma_\omega$, and we call $\omega$ the angle of accretivity in this case.
Then the optimal angle of accretivity for $\mathfrak a$ is
\begin{equation*}
  \omega(\mathfrak a) \coloneqq \inf\Bigl\{ \theta \in [0,\sfrac\pi2] \colon N(\mathfrak a) \subseteq \Sigma_\theta\Bigr\}.
\end{equation*}
The particular form that we consider will be that defining an elliptic operator in divergence form.

\begin{theorem}\label{thm:form}
  Let $\emptyset \neq \Omega \subseteq \R^d$ be any open set and let $\mu\colon \Omega \to \cL(\C^{d})$ be bounded and Lebesgue measurable.
  Let $V \subseteq W^{1,2}(\Omega)$ be a subspace.
  Put, for almost all $x \in \Omega$, $m(x) \coloneq \min_{\abs{\xi}=1} \Re(\mu(x)\xi,\xi)$, and, if $m(x) \geq 0$ for almost every $x \in \Omega$,
  \begin{equation*}
    \omega(\mu) \coloneq \inf\Bigl\{\theta \in [0,\sfrac\pi2] \colon N(\mu(x)) \subseteq \Sigma_\theta~\text{for almost every}~x \in \Omega\Bigr\}.
  \end{equation*}
  Further, consider the form $\ft_V$ on $L^2(\Omega)$ with $\dom(\ft_V) = V$ given by
  \begin{equation*}%
    \ft_V(u,v) \coloneq \int_\Omega (\mu\nabla u, \nabla v) \dd x.
  \end{equation*}%
  \begin{enumerate}[(i)]
  \item Suppose that $m(x) \geq 0$ for almost every $x \in \Omega$.
    Then $\ft_V$ is $\omega(\mu)$-accretive, so $\omega(\ft_V) \leq \omega(\mu)$.\label{item:t-form-form}
  \item Suppose that $m(x) > 0$ for almost every $x \in \Omega$, and that
    \begin{equation}\label{eq:locCoercitiv} \tan(\alpha) \coloneqq \esssup_{x \in \Omega} \frac
      {n(\ImOp{\mu(x))}}{m(x)} < \infty.
    \end{equation}
    Then $\omega(\mu) \leq \alpha<\sfrac\pi2$.\label{item:t-form-mu}
  \end{enumerate}\end{theorem}
\begin{proof}
  For~\ref{item:t-form-form}, we first note that the infimum in the definition of $\omega(\mu)$ is attained: Choosing a sequence $\theta_n \downarrow \omega(\mu)$ and collecting the corresponding exceptional null sets, there is a subset $\Lambda \subseteq \Omega$ with $\lvert \Omega\setminus\Lambda\rvert = 0$ such that $N(\mu(x)) \subseteq \bigcap_n \Sigma_{\theta_n} = \Sigma_{\omega(\mu)}$ for all $x \in \Lambda$, that is,
  \begin{equation*}
    (\mu(x)\xi,\xi) \in \Sigma_{\omega(\mu)} \qquad (x \in \Lambda,~\xi \in \C^d).
  \end{equation*}
  Thus, for all $u \in V$, we have $(\mu(x)\nabla u(x),\nabla u(x)) \in \Sigma_{\omega(\mu)}$ for almost every $x \in \Omega$, and since $\Sigma_{\omega(\mu)}$ is a closed convex cone, also $\ft_V(u,u) \in \Sigma_{\omega(\mu)}$.
  This shows $N(\ft_V) \subseteq \Sigma_{\omega(\mu)}$.
  Assertion~\ref{item:t-form-mu} follows directly from the definition of $\omega(\mu)$ and Lemma~\ref{lem:numrange-newer} applied to $L \coloneqq \mu(x)$ for almost every $x \in \Omega$.
\end{proof}

\begin{remark}\label{rem:some-sector-rules}
  The value of $\alpha$ in~\eqref{eq:locCoercitiv} is invariant under a multiplicative scalar perturbation $\rho$ of the coefficient matrix $\mu$, where $\rho$ is a real, bounded and strictly positive function on $\Omega$.
  Such a uniform sectoriality property is often of particular interest for nonautonomous parabolic or quasilinear problems, see for example \emph{Hypothesis~1} in the seminal work~\cite{AcquistapaceTerreni1987}, or also~\cite[Theorem~6.1]{Grisvard1969}.
\end{remark}

In a most useful turn of events, if $V$ is dense in $L^2(\Omega)$, then the numerical range inclusion for $\ft_V$ also transfers to the operator on $L^2(\Omega)$ induced by $\ft_V$.
To appreciate consequences of this, we first introduce some concepts from functional analysis.

\section{Sectorial operators, holomorphic semigroups, functional calculus}\label{sec:defin-prel-results}

We consider a \emph{reflexive} Banach space $X$ in all that follows.
This will not be a restriction since all concrete function spaces considered later will be in the reflexive range.
As before, we use $\Sigma_\theta$ for the closed sector of semi-angle $\theta$, and we denote its interior by $\Sigma_\theta^\circ \coloneqq {(\Sigma_\theta)}^\circ$.
It will be useful to have in mind that
\begin{equation*}
  \C\setminus\Sigma_\theta = - \Sigma_{\pi-\theta}^\circ.
\end{equation*}
We also agree on $\Sigma_0 \coloneq [0,\infty)$.
Also, recall that the symbol $\rho(B)$ stands for the resolvent set of a linear operator $B$ in $X$, that is, the set of all $\lambda \in \C$ for which $B-\lambda$ has a continuous inverse $X \to X$.
Accordingly, $\sigma(B) \coloneq \C\setminus\rho(B)$ is the spectrum of $B$.
\begin{definition}[Sectorial operator]\label{d-Winkelsec} Let $B$ be an operator in $X$ with dense domain.
  We call $B$ \emph{sectorial} of angle $\phi \in [0,\pi)$ if $\sigma(B)\subseteq \Sigma_\phi$ and
  \begin{equation*}
    \sup_{\lambda \in \C\setminus \Sigma_{\phi'}} \bigl\lVert\lambda(B - \lambda)^{-1} \bigr\rVert_{X\to X} < \infty \qquad \text{for every}~\phi' \in (\phi,\pi).
  \end{equation*}
  Moreover, we denote by $\phi(B) \coloneqq \inf \{\phi \in [0,\pi) \colon B~\text{is sectorial of angle}~\phi\}$ the \emph{spectral angle} for the operator $B$.
\end{definition}

Since we will work only in reflexive Banach spaces, it will not be a restriction to require sectorial operators to be densely defined.
Moreover, in this setting, a sectorial operator is injective if and only if its range is dense.
See~\cite[Proposition~2.1.1]{Haase2006} for these properties.

Sectorial operators of spectral angle $< \sfrac\pi2$ are precisely the (negative) generators of bounded holomorphic semigroups, and there is a direct relation between the spectral angle and the holomorphy angle of the semigroup, see~\cite[Theorem~3.7.11]{ArendtEtAl2001}.
We say that a strongly continuous semigroup $T \colon [0,\infty) \to \mathcal L(X)$ is \emph{holomorphic of angle $\vartheta \in (0,\sfrac\pi2]$} if $T$ has a holomorphic extension to the sector $\Sigma_\vartheta^{\circ}$ which is bounded on $\Sigma_\theta \cap \{z \in \C \colon |z| \leq 1\}$ for every $\theta \in (0,\vartheta)$, and \emph{bounded holomorphic of angle $\vartheta$} if the extension is bounded on $\Sigma_\theta$ for every $\theta \in (0,\vartheta)$.
The angle $\vartheta$ is the \emph{holomorphy angle} for the semigroup $T$.

Now, the aim of the next considerations is to introduce instruments that allow one to deduce sectoriality and more from properties of the numerical range of the operator in question.
For convenience, we state the following directly for Banach spaces for which we recall the definition of the numerical range and, associated, the duality mapping first.

Let $J \colon X \rightrightarrows X^*$ be the duality mapping of $X$, that is, the \emph{a priori} set-valued mapping
\begin{equation*}
  J(u) \coloneqq \Bigl\{u^* \in X^* \colon u^*(u) = \|u\|_X,~\|u^*\|_{X^*} = 1\Bigr\}.
\end{equation*}
By the Hahn-Banach theorem, $J(u) \neq \emptyset$ for every $u \in X$.
If $X^*$ is strictly convex, then $J(u)$ is in fact single valued for $u \neq 0$.
We will always be in this situation when we use the numerical range below for the family $X = L^p(\Omega)$ in the reflexive range $1<p<\infty$.
Let $B$ be a closed operator in $X$.
Then the \emph{numerical range} of $B$ is given by
\[
  N(B) \coloneqq \Bigl\{ u^*(Bu) \colon u \in \dom(B ),~\|u \|_X=1,~u^* \in J(u)\Bigr\}.
\]
This definition coincides with the one given before if $X$ is a Hilbert space via the Riesz representation theorem.

\begin{definition}[Accretive]\label{def:accretive-etc}
  Let $B$ be a closed operator in $X$.
  We say that $B$ is \emph{accretive} if $N(B) \subseteq [\Re z \geq 0] = \Sigma_{\sfrac\pi2}$.
  Further, $B$ is \emph{$m$-accretive} if it is accretive and there is a $\lambda < 0$ such that $\lambda \in \rho(B)$.
  Finally, we call $B$ \emph{$m$-$\omega$-accretive} for $\omega \in [0,\sfrac\pi2]$ if it is $m$-accretive and $N(B) \subseteq \Sigma_\omega$.
\end{definition}

We say that $T$ is a semigroup of contractions if $\norm{T(t)}_{X\to X} \leq 1$ for all $t \geq 0$ and analogously for holomorphic semigroups on the corresponding sectors.
Since $X$ is reflexive, by the Lumer-Phillips theorem, $m$-accretive operators are precisely the negative generators of strongly continuous semigroups of contractions~\cite[Theorem~3.4.5]{ArendtEtAl2001}, and for $\omega < \sfrac\pi2$, $m$-$\omega$-accretive operators will correspond to \emph{holomorphic} semigroups of contractions.
We prove this using the following fundamental and well-known proposition which gives decay of the resolvent in relation to distance to the numerical range.
See~\cite[Proposition~1.3.1]{MartinezSanz2001}.
For context for the following estimate, note that, given a set $\mathcal O \subseteq \C$ and $\lambda \in \C\setminus \overline{\mathcal O}$, we have $\sup_{z \in \mathcal O} \abs{\frac1{z-\lambda}} = \frac1{\dist(\lambda,\mathcal O)}$.

\begin{proposition}\label{p-resolvDECAY}
  Let $B$ be an operator in $X$.
  Let $U$ be a connected component of $\C \setminus \overline{N(B)}$.
  If $U \cap \rho(B) \neq \emptyset$, then $U \subseteq \rho(B)$, and
  \[
    \bigl\lVert(B - \lambda)^{-1}\bigr\rVert_{X\to X} \leq \frac{1}{\dist(\lambda, N(B))} \qquad (\lambda \in U).
  \]
\end{proposition}

\begin{proposition}\label{prop:sector-resolvent-general}
  Let $B$ be an operator in the reflexive Banach space $X$ and let $\omega \in [0,\sfrac{\pi}{2})$.
  Then the following are equivalent:
  \begin{enumerate}[(i)]
  \item $B$ is $m$-$\omega$-accretive.
  \item $-B$ is the generator of a holomorphic semigroup of contractions of angle $\sfrac\pi2 - \omega$.
  \end{enumerate}
  In either case, we have $\sigma(B) \subseteq \Sigma_\omega$ and for every $\theta \in (\omega,\sfrac\pi2]$,
  \begin{equation}\label{eq:sinexpress098}
    \bigl\lVert\lambda{(B - \lambda )}^{-1}\bigr\rVert_{X\to X} \leq \frac{1}{\sin
      (\theta-\omega)}
    \qquad
    (\lambda \in
    \C \setminus \Sigma_{\theta}).
  \end{equation}
  In particular, $B$ is sectorial with spectral angle $\phi(B) \le \omega$.
\end{proposition}

\begin{proof}%
  Suppose first that $B$ is $m$-$\omega$-accretive.
  Let $U$ be the connected component of $\C \setminus \overline{N(B)}$ that contains $\C \setminus \Sigma_\omega$.
  Then it follows from Proposition~\ref{p-resolvDECAY} that $\sigma(B) \subset \Sigma_\omega$ and that $\lVert{(B - \lambda)}^{-1}\rVert_{X\to X} \leq \frac1{\dist(\lambda, \Sigma_\omega)}$ for $\lambda \in \C\setminus\Sigma_\omega$.
  Now let $\theta \in (\omega,\sfrac\pi2]$ and $\lambda \in \C\setminus \Sigma_{\theta}$.
  Suppose first that $\lambda \in \C\setminus \Sigma_{\sfrac\pi2+\omega}$.
  Then $\dist(\lambda,\Sigma_\omega) = \abs{\lambda}$ because $0$ is the closest point in $\Sigma_\omega$ to $\lambda$.
  (Make a sketch.)
  On the other hand, if $\lambda \in \Sigma_{\sfrac\pi2 + \omega} \setminus\Sigma_{\theta}$, then with a bit of elementary geometry we find
  \[ \dist(\lambda, \Sigma_\omega) = |\lambda|\, \sin\bigl( |\arg(\lambda)| - \omega\bigr) \geq |\lambda| \, \sin(\theta-\omega)\] and~\eqref{eq:sinexpress098} follows.
  Since $X$ is reflexive, $B$ is densely defined~\cite[Proposition~2.1.1]{Haase2006}, so $B$ is sectorial with $\phi(B) \leq \omega$.
  Thus, with~\cite[Theorem~3.7.11]{ArendtEtAl2001} it follows that the semigroup $S$ generated by $-B$ is bounded holomorphic with holomorphy angle $\sfrac\pi2 - \omega$.
  Now let $z \in \Sigma_{\sfrac\pi2-\omega}^{\circ}$ and write $S(z) = S(t e^{\ii \varphi}) \coloneqq S_{\varphi}(t)$ with $\varphi = \arg{z} \in (\omega-\sfrac\pi2,\sfrac\pi2-\omega)$ and $t = \abs{z}$.
  With~\cite[Proposition~3.7.2]{ArendtEtAl2001}, we find that the negative generator of $S_{\varphi}$ is $e^{\ii \varphi}B$ which is still $m$-accretive.
  So $S_{\varphi}$ is a semigroup of contractions by the Lumer-Phillips theorem~(\cite[Theorem~3.4.5]{ArendtEtAl2001}), and the assertion follows.

  Vice versa, suppose that $B$ is the negative generator of a holomorphic semigroup of contractions of angle $\sfrac\pi2-\omega$.
  Then $B$ is sectorial, in particular, $(-\infty,0) \subset \rho(B)$.
  Let $\omega - \sfrac\pi2 < \varphi < \sfrac\pi2-\omega$.
  Then the semigroups $S_{\varphi}$ are still semigroups of contractions and so their negative generators $e^{\ii \varphi}B$ are $m$-accretive by the Lumer-Phillips theorem.
  Let $u \in \dom(B)$ with $\norm{u}_X = 1$ and let $u^{*} \in J(u)$.
  Then $u^{*}(e^{\ii \varphi}Bu) \in \Sigma_{\sfrac\pi2}$ for all $\varphi$ as before.
  So $u^{*}(Bu) \in \Sigma_{\theta}$ for all $\theta > \omega$.
  Then $u^{*}(Bu) \in \Sigma_{\omega}$ and $N(B) \subseteq \Sigma_{\omega}$ follows.
\end{proof}

We next review a functional calculus for injective sectorial operators $B$ in $X$.
Recall that in this case the range of $B$ is dense in $X$.

The notion of interest concerns the Dunford-Riesz functional calculus for realizing $f(B)$ for certain classes of holomorphic functions on an open sector $\Sigma_\theta^\circ$ with $\theta > \phi(B)$, the spectral angle of $B$.
We will stay concise and refer to~\cite[Sections~1.4, 2.1 and~2.4]{DenkHieberPruess2003} for the details, also~\cite{BattyHaaseMubeen2013,Haase2006}.
The functional calculus is based on the Dunford integral to define $f(B)$ and yields a bounded algebra homomorphism $\Phi_B \colon \cH^\infty_0(\Sigma_\theta^\circ) \to \cL(X)$, taking $f \mapsto f(B)$, where $\theta > \phi(B)$ and
\begin{equation*}
  \cH^\infty_0(\Sigma_\theta^\circ) \coloneqq \Bigl\{ f\colon \Sigma_\theta^\circ \to \C~\text{holomorphic}
  \colon \exists c \geq 0,~s>0\colon \lvert f(z) \rvert \leq c \min(\lvert z\rvert^s,{\lvert z\rvert}^{-s})\Bigr\}.
\end{equation*}
Evidently, functions in $\cH^\infty_0(\Sigma_\theta^\circ)$ are required to decay in a regular way at zero and infinity.
However, often, it is of particular interest to extend the foregoing functional calculus also to the class of \emph{bounded} holomorphic functions on $\Sigma_\theta^\circ$,
\begin{equation*}
  \cH^\infty(\Sigma_\theta^\circ) \coloneqq \Bigl\{ f\colon \Sigma_\theta^\circ \to \C
  \colon f~\text{holomorphic and bounded}\Bigr\}.
\end{equation*}

\begin{definition}[$\cH^\infty$-calculus]
  Let $B$ be an injective sectorial operator in $X$ and let $\theta > \phi(B)$, where $\phi(B)$ is the spectral angle of $B$.
  Then $B$ is said to admit a \emph{bounded $\cH^\infty(\Sigma_\theta^\circ)$-calculus} if there is a constant $C_\theta \geq 0$ such that
  \begin{equation*}
    \bigl\lVert f(B) \bigr\rVert_{X\to X} \le C_\theta \sup_{z \in \Sigma_\theta^\circ} \bigl\lvert f(z)\bigr\rvert \qquad (f \in \cH^\infty_0(\Sigma_\theta^\circ)).
  \end{equation*}
  In this case, $\Phi_B$ extends uniquely to a bounded algebra homomorphism on $\cH^\infty(\Sigma_\theta^\circ)$~\cite[Proposition~5.3.4]{Haase2006}.
  We shall denote the infimum of all $\theta$ for which $B$ admits a bounded $\cH^\infty(\Sigma_\theta^\circ)$-calculus by $\phi_{\infty}(B)$, and we call it the \emph{$\cH^\infty$-angle} for $B$.
  Often, we just say that $B$ admits a bounded $\cH^\infty$-calculus (on $X$).
\end{definition}

Having an operator with a bounded $\cH^\infty$-calculus at hand is desirable for a quite diverse zoo of problems, see e.g.~\cite{CubillosJimenez2025,DieballGodec2023,HassellVasy2000,KrieglerWeis2018,KurbatovKurbatova2012,LindemulderVeraar2020,OlivaMazaWarma2023,SanchezEtAl2023,SchnaubeltVeraar2017}.
In particular, a bounded $\cH^\infty$-calculus with $\cH^\infty$-angle $\phi_{\infty}(B) < \sfrac\pi2$ implies \emph{maximal parabolic regularity} for the operator $B$ via bounded imaginary powers, at least if the corresponding Banach space is a UMD~space.
(Such as $W^{k,p}(\Omega)$ for $1 < p < \infty$ and $k \in \N_0$.)
We refer to the seminal Dore-Venni paper~\cite{DoreVenni1987}, see also~\cite{LeMerdy1998} and~\cite[Section~9.3.3]{Haase2006}.
Maximal parabolic regularity is a most desirable property in the treatment of nonlinear evolution equations, see~\cite{Pruess2002,KunstmannWeis2004} or~\cite[Chapter~III]{Amann1995}, for example.
Further, from the $\cH^\infty$-calculus we recover sectoriality of $B$---which is built into the definition anyway---and associated estimates as well as the generator property for a holomorphic semigroup by choosing $f(z) = \frac1{z-\lambda}$ for $\lambda$ outside of a sector, and $f(z) = e^{-tz}$, for $t>0$, for example.
A bounded $\cH^{\infty}$-calculus is thus \emph{the strongest} of these three and an estimate on the $\cH^{\infty}$-angle implies estimates on the spectral angle and, equivalently, the holomorphy angle of the semigroup.

In the special case of a Hilbert space, we have the following general result, see~\cite[Corollary~10.12]{KunstmannWeis2004} together with~\cite[Corollary~7.1.17]{Haase2006}.
The latter contains the seminal Crouzeix-Delyon theorem which even provides an explicit and uniform norm estimate for the functional calculus which is of independent interest, cf.~\cite{CrouzeixDelyon2003}.
\begin{proposition}%
  \label{prop:m-accretive-h-infty}
  Let $\omega \in [0,\sfrac\pi2]$ and suppose that $B$ is an injective $m$-$\omega$-accretive operator in a Hilbert space $X$.
  Then $B$ admits a bounded $\cH^\infty$-calculus on $X$ with $\phi_{\infty}(B) = \phi(B) \leq \omega$.
\end{proposition}

\section{The operator in \texorpdfstring{$L^2(\Omega)$}{L²(Ω)}}\label{sec:operator-l2omega}

We next collect various consequences of the accretivity of $\ft_V$ as in Theorem~\ref{thm:form}.
For this purpose, let us fix the following setting for all that follows.
We will add additional assumptions on $V$ and the geometry of $\Omega$ eventually in later sections.

\begin{assumption}[Standing assumptions]\label{a-lowerbound}%
  Suppose that $\Omega \subseteq \R^d$ is an open set and let $\mu\colon \Omega \to \cL(\C^{d})$ be bounded and Lebesgue measurable.
  \begin{enumerate}[(i)]
  \item Let $V \subseteq W^{1,2}(\Omega)$ be a closed subspace that contains $C_c^{\infty}(\Omega)$.
  \item Put $m(x) \coloneq \min_{\abs{\xi}=1} \Re(\mu(x)\xi,\xi)$ and suppose that $m_{\bullet} \coloneq \essinf_{x\in\Omega} m(x) > 0$.
  \end{enumerate}
\end{assumption}

Recall the form
\begin{equation*}%
  \ft_V(u,v) \coloneq \int_\Omega (\mu\nabla u, \nabla v) \dd x, \qquad \dom(\ft_{V}) = V.
\end{equation*}
With the standing assumption that $V$ contains $C_c^{\infty}(\Omega)$, it follows that $V$ is dense in $L^2(\Omega)$ we get a well defined operator $A$ induced by $\ft_V$ on $L^2(\Omega)$, which will be the basis for all further considerations:
\begin{equation}
  \begin{split}\dom(A) & \coloneqq \Bigl\{u \in V \colon \exists\, f \in L^2(\Omega) \colon \ft(u,v) = (f,v)_{L^2(\Omega)}~\text{for all}~v\in V\Bigr\} \\ Au &\coloneqq f.\end{split}\label{eq:L2-operator}
\end{equation}
It is quickly verified that $A$ is a closed operator in $L^2(\Omega)$, since $V$ is a closed subspace of $W^{1,2}(\Omega)$ and $\mu$ is uniformly elliptic by Assumption~\ref{a-lowerbound}.
See Remark~\ref{rem:closed-subspace-or-not} below, though.
With Theorem~\ref{thm:form}, the standing assumptions imply that the form $\ft_V$ is $\omega(\mu)$-accretive.
Since by definition, $(Au,u)_{L^2(\Omega)} = \ft_V(u,u)$ for all $u \in \dom(A) \subseteq V$, it is clear that $N(A) \subseteq N(\ft_V)$.
In fact, $N(A)$ is dense in $N(\ft_V)$ by~\cite[Corollary~VI.2.3]{Kato1980}, so we have:
\begin{equation}
  \label{eq:sector-form-and-operator}
  \omega(A) = \omega(\ft_V) \leq \omega(\mu) \leq \alpha < \sfrac\pi2%
\end{equation}
with $\alpha$ as in~\eqref{eq:locCoercitiv}.
In the following, we will derive several results and angles which will depend on $\omega(\ft_V)$ and $\omega(\mu)$.
This will give the most precise and sharpest results, yet of course it is of importance to have a good upper bound like $\alpha$ at our disposal, for example when we are interested in \emph{uniformity} of associated sectors.
Such a situation was already mentioned in Remark~\ref{rem:some-sector-rules}.
In any of the following, we can readily substitute the upper bound $\alpha$ to get a concrete upper bound of the desired quantity.
We proceed along the following settings:
\begin{enumerate}[(i)]
\item Considering $A$ in $L^2(\Omega)$, where $\omega(\ft_V)$ directly determines the resulting quantities.
\item The realizations of $A$ in $L^p(\Omega)$:
  \begin{enumerate}[(a)]
  \item The general case, where the results are stated in some (explicit) dependence on $\omega(\mu)$ and $p$, and
  \item peculiar geometric situations which yield $p$-independent results, so the situation for $p=2$ is decisive.\end{enumerate}
\item Functional spaces other than $L^p(\Omega)$, in particular negative (dual) Sobolev spaces.
\end{enumerate}

We will thereby be mostly interested in the case where $V$ corresponds to a subspace of $W^{1,2}(\Omega)$ that carries mixed boundary conditions in the usual sense.
Since this will already be a comprehensive endeavor, we will not pursue the manifold possible extensions, such as considering Robin boundary conditions or dynamical boundary conditions.

To start with the foregoing agenda, from the standing assumption, the Lax-Milgram lemma implies that the whole (open) left half plane $[\Re z < 0]$ belongs to the resolvent set $\rho(A)$ of $A$.
As per~\eqref{eq:sector-form-and-operator}, we further have $N(A) \subseteq \Sigma_{\omega(\ft_{V})}$ with $\omega(\ft_V) < \sfrac\pi2$.
In particular, $A$ is $m$-$\omega(\ft_V)$-accretive which immediately gives the sectoriality of $A$ with the estimate~\eqref{eq:sinexpress098} for the generator property for a holomorphic semigroup of contractions for $-A$ via Proposition~\ref{prop:sector-resolvent-general}.
Proposition~\ref{prop:m-accretive-h-infty} further implies a bounded $\cH^{\infty}$-calculus for injective versions of $A$.
Here, we call $\ft_V$ \emph{coercive} if there is $c > 0$ such that $\Re \ft_V(u,u) \geq c \norm{u}_{W^{1,2}(\Omega)}^2$ for all $u \in V$; if $\ft_V$ is coercive, then $A$ is injective, see also Remark~\ref{rem:inject}.
Thus:

\begin{theorem}\label{thm:analytHG}
  Let $A$ be the operator induced by $\ft_V$ on $L^2(\Omega)$.
  Then $A$ is $m$-$\omega(\ft_V)$-accretive.
  In particular, it is sectorial with $\phi(A) \leq \omega(\ft_V)$ and for every $\theta \in (\omega(\ft_{V}),\sfrac\pi2]$, we have the explicit estimate
  \begin{equation*}%
    \bigl\lVert{\lambda(A - \lambda)}^{-1}\bigr\rVert_{L^2(\Omega)\to L^2(\Omega)}\leq \frac{1}{\sin
      (\theta-\omega(\ft_{V}))}%
    \qquad (\lambda \in
    \C \setminus \Sigma_{\theta}).
  \end{equation*}
  Moreover, $-A$ is the generator of a holomorphic semigroup of contractions on $L^2(\Omega)$ with the holomorphy angle $\sfrac\pi2 - \omega(\ft_V)$, and for every $\delta > 0$, the operator $A + \delta$ admits a bounded $\cH^\infty$-calculus with $\cH^\infty$-angle $\phi_{\infty}(A+\delta) = \phi(A+\delta) \leq \omega(\ft_V)$.
  If $\ft_V$ is coercive, then the foregoing also holds true for $\delta =0$.
\end{theorem}

\begin{remark}\label{rem:inject}
  Coercivity of $\ft_V$ makes $A$ injective, which is the desired property.
  The form $\ft_V$ is coercive whenever one has a Poincar\'e inequality for $V$.
  For mixed boundary conditions as incorporated in $V$, there is still a Poincar\'e inequality even for very irregular geometries.
  See the detailed discussion in~\cite[Section~7]{EgertHallerDintelmannRehberg2015}.
  In any case, a scalar shift $\delta>0$ can rectify the issue at hand and the resulting sector inclusions for form and operator are still as before.
\end{remark}

\begin{remark}%
  \label{rem:closed-subspace-or-not}
  We mention that Theorem~\ref{thm:form} does not require a uniform ellipticity constant for $\mu$ or a closed form domain $V$.
  We could still obtain a well defined operator $A$ associated with the form $\ft_V$ without these properties, as long as $V$ is dense in $L^2(\Omega)$, that even turns out to be $m$-$\omega(\ft_V)$-accretive and the negative generator of a holomorphic semigroup.
  See the seminal work~\cite{ArendtTerElst2012}.
  We stick with the standing Assumption~\ref{a-lowerbound} in order not to complicate the matter at hand, since most of the results we use and quote in the following are written in the literature for uniformly elliptic coefficients and closed form domains.
\end{remark}

\section{The operator in \texorpdfstring{$L^p(\Omega)$}{Lᵖ(Ω)}}\label{sec:operator-lpomega}

In recent years it has turned out that for the functional analytic treatment of many real world problems, the $L^2(\Omega)$-calculus for the corresponding differential operators is not an adequate one, but rather one anchored in an $L^p(\Omega)$ space is so.
This usually happens if one is confronted with irregular data of the problem (non-smooth, non-convex domain, or discontinuous coefficients, or mixed boundary conditions), because then one cannot rely on elliptic $H^2(\Omega)$ regularity, or if one has to deal with certain nonlinear equations, for example in particular \emph{quasilinear} ones, or semilinear problems without a monotonicity assumption.

Prototypes for such problems could be the classical semiconductor equations (van Roosbroeck system,~\cite{KaiserNeidhardtRehberg2009}) or the Keller-Segel model from mathematical biology~(\cite{HorstmannMeinlschmidtRehberg2018}); see also~\cite[Section~5]{MazyaEtAl2004} for a general $L^p$-framework for quasilinear reaction-diffusion systems including nonlocal terms.
Here, let us point out that while one might expect from this brief motivation that when going over to an $L^p(\Omega)$ setting, \emph{surely} one would consider $p>2$.
But in fact, it turns out that it indeed can be adequate to consider $p<2$, as done for example in~\cite{HorstmannMeinlschmidtRehberg2018}.
This case is also of elevated interest in harmonic analysis, for example for boundedness of the Riesz transform and associated questions such as Kato square root estimates, see e.g.~\cite{AuscherEtAl2015}.

In contrast to the $L^2(\Omega)$ setting, for the operator $A$ considered in $L^p(\Omega)$ it makes quite a fundamental difference whether the coefficient function $\mu$ is real- or complex valued.
We first consider the case of real coefficients where one can derive quite satisfying results under very few structural assumptions.
In any case, we specialize to subspaces $V$ of $W^{1,2}(\Omega)$ which incorporate a Dirichlet type boundary condition on a closed subset $D$ of the boundary $\partial\Omega$.
To this end, in addition to Assumption~\ref{a-lowerbound} we activate \emph{from now on} the following assumption:
\begin{assumption}\label{a-standard}
  Let $D \subseteq \partial\Omega$ be closed, and let $V$ be the closure in $W^{1,2}(\Omega)$ of either
  \begin{equation}\label{eq:subspace1}
    C^\infty_D(\Omega) \coloneq\Bigl\{u|_{\Omega}\colon u \in C^\infty_c(\R^d),~\dist(\supp (u), D) >0 \Bigr\}
  \end{equation}
  or of
  \begin{equation}\label{eq:closureII}
    \underline{W}^{1,2}_D(\Omega) \coloneqq \Bigl\{u \in W^{1,2}(\Omega) \colon \dist (\supp (u), D) >0 \Bigr\}.
  \end{equation}
\end{assumption}

Note that we have $C_c^\infty(\Omega) \subset V$ in both cases.
When $D = \emptyset$, we clearly recover $\underline{W}^{1,2}_\emptyset(\Omega) = W^{1,2}(\Omega)$.
If $D \neq \emptyset$, then either space $V$ carries a Dirichlet boundary condition for its elements on $D$, at least implicitly so.
In how far this condition can be made more explicit depends on the regularity of $\partial\Omega$ around $D$, see e.g.~\cite{ChillMeinlschmidtRehberg2021,EgertTolksdorf2017}.

Let $A$ be the operator in $L^2(\Omega)$ which is induced by $\ft_V$ as in~\eqref{eq:L2-operator}.
The convenient way to define the realizations in $L^p(\Omega)$ works by extrapolation of the semigroup $T$ generated by $-A$ on $L^2(\Omega)$.
For this purpose, let us introduce (the interior of) the bounded extrapolation range of $T$:
\begin{equation*}
  \cI(A) \coloneqq \inte \Bigl\{ p \in (1,\infty) \colon \sup_{t>0}\,\lVert T(t)\rVert_{L^p(\Omega)\to L^p(\Omega)} < \infty\Bigr\}.%
\end{equation*}
We point out that $\cI(A)$ is open by definition.
From Riesz-Thorin interpolation, it is also clear that $\cI(A)$ is always an interval.
For $p \in \cI(A)$, the semigroup $T$ extends uniquely from $L^p(\Omega) \cap L^2(\Omega)$ to a strongly continuous semigroup $T_p$ on $L^p(\Omega)$, and we can thus \emph{always} define $A_p$ to be the negative generator of $T_p$; in particular, $A_2 = A$ and $T_2 = T$.
Moreover, holomorphy of the semigroup $T$ implies that $T_p$ is bounded holomorphic by Stein's extrapolation theorem~\cite[p.~96]{Ouhabaz2005}, which also gives a crude estimate on the holomorphy angle.
It is thus clear that $\cI(A)$ will play an important role from now on.
We will next derive estimates in terms of the coefficient function.

\textbf{Real coefficients.}
We first consider the case where the coefficient function $\mu$ is real, that is, takes its values in $\cL(\R^d)$.
Strictly speaking, this case is contained in the one of complex coefficients treated below: a real $\mu$ is $p$-elliptic for every $p \in (1,\infty)$, and the angle in Corollary~\ref{cor:complex-operator} then coincides with the one in Theorem~\ref{thm:chill}.
We nevertheless treat it separately, since for real $\mu$ the semigroups are sub-Markovian, in particular positive, which we need for the $\cH^\infty$-angle in Theorem~\ref{thm:functkalc}.
We say that $T$ is \emph{positive} if $T(t)f$ is nonnegative for all nonnegative $f \in L^2(\Omega)$ and all $t > 0$.
Clearly, $T$ must necessarily be real---mapping real functions into real functions---in order to be positive; this is the case, for instance, if the coefficient function is real.
Further, $T$ is \emph{$L^\infty$-contractive} if
\begin{equation*}
  \|T(t)f\|_{L^\infty(\Omega)} \leq \|f\|_{L^\infty(\Omega)} \qquad (f \in L^2(\Omega) \cap L^\infty(\Omega),\quad t>0).
\end{equation*} A positive and $L^\infty$-contractive semigroup is also called \emph{sub-Markovian}.

\begin{proposition}\label{p-basicsform}
  Let $\mu$ be real.
  Then the operator $A$ on $L^2(\Omega)$ induced by $\ft_V$ is the negative generator of a sub-Markovian, holomorphic semigroup $T_2$ of contractions on $L^2(\Omega)$, and we have $\cI(A) = (1,\infty)$.
  More precisely, the semigroup extrapolates consistently to a holomorphic positive semigroup of contractions $T_p$ on $L^p(\Omega)$ for every $p \in (1,\infty)$.
\end{proposition}

The individual statements in Proposition~\ref{p-basicsform} are scattered in~\cite{Ouhabaz2005} at various places, but collected in detail in~\cite[Section~2.2]{ChillMeinlschmidtRehberg2021}.
Here we have already used several structural properties of the particular choices of $V$ as in~\eqref{eq:subspace1} and~\eqref{eq:closureII}, and that $\mu$ is real.

By Stein interpolation~\cite[p.~96]{Ouhabaz2005}, $T_p$ is a holomorphic semigroup of contractions on some sector, so $A_p$ is an $m$-$\omega$-accretive operator in $L^p(\Omega)$ for some $\omega < \sfrac\pi2$ by Proposition~\ref{prop:sector-resolvent-general}, and it is not hard to see that if $p >2$ and $\abs{\Omega} < \infty$, then $A_p$ is the part of $A$ in $L^p(\Omega)$.
Without further assumptions on $\Omega$ and its boundary, we do not recover the angle of accretivity of the $L^2(\Omega)$-operator.
See Section~\ref{sec:more-assumpt-geom}, though.
Recall that $\omega(\mu)$ is the optimal angle of accretivity for the coefficient function $\mu$.
Then:
\begin{theorem}[{\cite[Theorem~3.4 and Corollary~3.7]{ChillMeinlschmidtRehberg2021}}]\label{thm:chill}
  Let $p \in (1,\infty)$ and let $\mu$ be real.
  Then $A_p$ is $m$-$\alpha_p$-accretive with
  \begin{equation}\label{eq:lpWinkel}
    \tan (\alpha_p) = \frac {\sqrt {(p-2)^2 + p^2 \tan( \omega(\mu))^{2}}}{2 \sqrt {p-1}}.
  \end{equation}
  In particular, the semigroup $T_p$ is holomorphic of angle $\sfrac\pi2-\alpha_{p}$ and the spectral angle $\phi(A_p)$ is not larger than $\alpha_p$.
\end{theorem}

We note that as a function of $\omega(\mu)$, the angle $\alpha_p$ cannot be improved in general, see~\cite[Remark~4]{ChillEtAl2006} and~\cite[Theorem~2]{ChillEtAl2005}.
Equivalently to~\eqref{eq:lpWinkel}, we have
\begin{equation*}
  \cos(\alpha_p) = \frac{2\sqrt{p-1}}{p}\cos(\omega(\mu)).
\end{equation*}
Clearly, $\alpha_p$ also collapses to $\alpha_2 = \omega(\mu)$ which is the natural upper bound for $\omega(\ft_V)$, the optimal angle of accretivity for $A$.
Observe however that $\alpha_p$ depends on the coefficient function $\mu$ representing $A$: an antisymmetric part of $\mu$ that does not change the operator, as for example in~(ii) of the closing comments, increases $\alpha_p$.
So whenever the symmetric part $\mu_s$ induces the same form as $\mu$, one should use $\mu_s$, for which $\tan(\alpha_p) = \frac{\abs{p-2}}{2\sqrt{p-1}}$.

Next, we relate the $\cH^\infty$-angle of $A+\delta$, which is at most $\omega(\ft_V)$, with the one of $A_p+\delta$ by interpolation.
Here we use that the semigroups $T_p$ are positive semigroups of contractions as established in Proposition~\ref{p-basicsform}.

\begin{theorem}\label{thm:functkalc}
  Let $p \in (1,\infty)$ and $\delta > 0$ and let $\mu$ be real.
  Then $A_p+\delta$ admits a bounded $\cH ^\infty$-calculus on $L^p(\Omega)$ with $\phi_{\infty}(A_{p}+\delta) \leq \frac {\pi}{2} |1-\frac {2}{p}|+ \omega(\ft_{V}) (1- |1-\frac {2}{p}|)$.
  If $\ft_V$ is coercive, then the assertion is also true for $\delta =0$.
\end{theorem}
\begin{proof}
  We combine an ingenious interpolation result for $\cH^\infty$-angles as in~\cite[Proposition~4.9]{KaltonKunstmannWeis2006} with the result of Duong~\cite{Duong1990} which says that for $p \in (1,\infty)$, each $A_p+\delta$ has a bounded $\cH^\infty(\Sigma_{\theta}^\circ)$-calculus for \emph{any} $\theta > \sfrac\pi2$.
  For a comprehensive proof, see also~\cite{HieberPruess1998}.
  We then interpolate between the angle $\omega(\ft_V)$ for $p=2$ and $\sfrac\pi2$, which is an upper bound for the $\cH^\infty$-angle $\phi_\infty(A_r+\delta)$ for all the other $r \in (1,\infty)\setminus\{2\}$.

  Let us start by verifying the assumptions of the interpolation result~\cite[Proposition~4.9]{KaltonKunstmannWeis2006}.
  Writing $B_p \coloneqq A_p + \delta$, these require that there is $\lambda \in \C$ with $\Re \lambda < 0$ such that for any $p,q \in (1,\infty)$, the resolvents in $\lambda$ of $B_p$ and $B_q$ are consistent, that is,
  \begin{equation*}
    (B_p - \lambda)^{-1}u = (B_q - \lambda)^{-1}u \qquad
    (u \in L^p(\Omega) \cap L^q(\Omega)).
  \end{equation*}
  But from Proposition~\ref{p-basicsform} we already know that the semigroups generated by $-B_p$ and $-B_q$, respectively, are consistent for all $p,q \in (1,\infty)$, and via the Laplace transform of the semigroups this indeed implies consistency of the resolvents.

  Now for the main proof.
  In view of Theorem~\ref{thm:analytHG}, there is nothing to do for $p=2$.
  Let $p \in (1,\infty) \setminus \{2\}$ be given.
  If $p<2$, let $r \in (1,p)$, otherwise, let $r \in (p,\infty)$.
  In either case there is $\theta \in (0,1)$ such that
  \[
    L^p(\Omega) = \Bigl[L^r(\Omega), L^2(\Omega)\Bigr]_\theta, \qquad \frac {1}{p} = \frac {1 -\theta}{r} + \frac {\theta}{2}.
  \]
  Then, by~\cite[Proposition~4.9]{KaltonKunstmannWeis2006}, also $\phi_\infty(A_p+\delta) \leq (1-\theta)\phi_\infty(A_r+\delta) + \theta{\omega(\ft_{V})}$.
  In particular, $\phi_\infty(A_p+\delta) \leq (1-\theta)\frac\pi2 + \theta \omega(\ft_{V})$.
  To get the best possible upper bound, we optimize (maximize) for $\theta$: for $p>2$, we let $r\uparrow\infty$, for which $\theta \uparrow \frac2p$, whereas for $p < 2$, we let $r \downarrow 1$, for which $\theta \uparrow 2-\frac2p$.
  Encoding the case distinction in an absolute value, this yields exactly that $\phi_\infty(A_p+\delta) \leq \frac {\pi}{2} |1-\frac {2}{p}|+ \omega(\ft_{V}) (1- |1-\frac {2}{p}|)$ for all $p \in (1,\infty)$, which was the claim.
\end{proof}

\begin{remark}\label{rem:hiebpruess}
  As observed in the proof, Theorem~\ref{thm:functkalc} in particular implies that $\phi_\infty(A_p+\delta) < \sfrac\pi2$ for $p \in (1,\infty)$.
  This also follows from abstract results which exploit intrinsic connections of a bounded $\cH^{\infty}$-calculus to $\mathcal R$-boundedness, see e.g.~\cite[Corollary~5.2]{KaltonWeis2001}, using that the semigroups $T_p$ are positive, holomorphic semigroups of contractions as stated in Proposition~\ref{p-basicsform}.
  While Theorem~\ref{thm:functkalc} is more explicit, the statement $\phi_\infty(A_p+\delta) < \sfrac\pi2$ is often of foremost interest due to its connection to maximal parabolic regularity via bounded imaginary powers~\cite{DoreVenni1987}.
\end{remark}

\textbf{Complex coefficients.}
We now consider the general case of a \emph{complex} coefficient function.
In comparison to the case of real coefficients, the situation regarding accretivity or sectoriality of $A$ in $L^p(\Omega)$, properties of the associated semigroup and a bounded $\cH^\infty$-calculus is (much) more intricate.

Indeed, even for the most straightforward case of a second-order operator with complex coefficients and a scalar shift, the technique employed in Proposition~\ref{p-basicsform} to extrapolate the semigroup $T$ generated by $-A$ to $L^p(\Omega)$ based on $L^\infty$-contractivity of $T$ does not work at all: in the case of complex coefficients, there just \emph{is} \emph{no} $L^\infty$-contractivity in general, and its failure is quite catastrophic in nature~\cite{AuscherEtAl2000,AuscherCoulhonTchamitchian1996,MazyaNazarovPlamenevskij1985}.
In fact, for complex coefficients, there are situations where even for self-adjoint differential operators, there is a finite $p$ for which the associated $L^2$-semigroup $T$ ceases to exist on $L^p(\Omega)$, see~\cite{Davies1997}.

Nevertheless, in recent years there was considerable progress in understanding and clarifying the situation.
A particularly useful concept is the following, popularized by~\cite{CarbonaroDragicevic2020}: For $p \in (1,\infty)$, we say that a matrix function $\mu \colon \Omega \to \cL(\C^{d})$ is \emph{$p$-elliptic} if $\Delta_p(\mu) > 0$ with
\begin{equation}
  \label{eq:p-elliptic}
  \Delta_p(\mu) \coloneqq \essinf_{x\in\Omega}\min_{\xi \in \C^d,|\xi|=1} \Re\bigl(\mu(x)\xi,\mathcal J_p(\xi)\bigr)
\end{equation}
where $\cJ_p$ is the $\R$-linear map $ \cJ_p(\alpha + \ii \beta) \coloneqq \frac{2\alpha}{p'} + \ii \frac{2\beta}{p}$.
Of course, $2$-ellipticity is exactly the usual uniform ellipticity of $\mu$ over $\Omega$.
Further, it is easily seen that $\Delta_p(\mu) = \Delta_{p'}(\mu)$, in particular, $\mu$ is $p$-elliptic if and only if it is $p'$-elliptic~\cite[Lemma~8]{Egert2020}, and $p \mapsto \Delta_p(\mu)$ is Lipschitz continuous, increasing on $(1,2]$ and decreasing on $[2,\infty)$~\cite[Corollary~5.16]{CarbonaroDragicevic2020}.
From the latter it also follows that $\{ p \in (1,\infty) \colon \mu~\text{is $p$-elliptic}\}$ is an open interval.
The significance of $p$-ellipticity is the following beautiful result of Egert~\cite{Egert2020} that we give in a convenient form for us.
Here we use the particular choices for $V$ as in Assumption~\ref{a-standard}.\footnote{We mention that the case of $V$ given by the closure of $\underline{W}^{1,2}_{D}(\Omega)$ as in~\eqref{eq:closureII} is not treated in~\cite{Egert2020} explicitly, but it is easily traced---and verified with the author---that the particular structural results established and required for~\cite[Theorem~2]{Egert2020} for $V$ given by the closure of $C_{D}^{\infty}(\Omega)$ as in~\eqref{eq:subspace1} stay valid.}
Recall that $T$ is the semigroup generated by $-A$ on $L^2(\Omega)$.

\begin{proposition}[{\cite[Theorem~2]{Egert2020}}]\label{prop:p-elliptic-extrapolate-general}
  Let $p \in (1,\infty)$ and suppose that $\mu$ is $p$-elliptic.
  Then $p \in \cI(A)$.
  In particular, $T$ extrapolates to a holomorphic semigroup of contractions $T_p$ on $L^p(\Omega)$.
\end{proposition}

Since the set of all $p$ for which $\mu$ is $p$-elliptic is open, \cite[Theorem~2]{Egert2020} applies in a neighborhood of $p$, so $p$ indeed belongs to the interior of the bounded extrapolation range, that is, to $\cI(A)$.

This foregoing proposition is quite satisfying from a theoretical point of view and will be the motivation to work precisely with $p$-elliptic coefficient functions.
Still, it is clearly of interest to have a direct sufficient condition on $p$-ellipticity in terms of matrix data at hand.
See Remark~\ref{rem:suff-cond-p-ellipt} for one.
In fact, let us mention that $\mu$ is $p$-elliptic for \emph{all} $p \in (1,\infty)$ \emph{if and only if} $\mu$ is real~\cite[Lemma~2.15]{TerElstEtAl2021}.
This yields another proof of $\cI(A) = (1,\infty)$ in the case of a real coefficient function.
For Proposition~\ref{p-basicsform}, we preferred to argue via the sub-Markov property because the $\cH^{\infty}$-angle result in Theorem~\ref{thm:functkalc} requires positivity of the semigroup.

So suppose that $\mu$ is $p$-elliptic for some $p \in (1,\infty)$.
Then $p \in \cI(A)$ by Proposition~\ref{prop:p-elliptic-extrapolate-general}, so $A_p$ is defined as before.
Since $p$-ellipticity is an \emph{open} condition, there is $r \in (1,\infty)$ with $\abs{\frac12-\frac1p} < \abs{\frac12-\frac1r}$ such that $\mu$ is still $r$-elliptic.
Again by Stein interpolation, now between $L^2(\Omega)$ and $L^r(\Omega)$, the semigroup $T_p$ is a holomorphic semigroup of contractions on some sector.
So the operators $A_p$ are $m$-$\omega$-accretive for some $\omega < \sfrac\pi2$ by Proposition~\ref{prop:sector-resolvent-general}.

We next present a new estimate on the angle of accretivity of the operators $A_p$ with complex coefficients.
Our proof separates the angle estimate for the coefficient matrix from the operator.
For this purpose, let us put $p \in (1,\infty)$ and introduce the \emph{$p$-numerical range} of $M \in \cL(\C^{d})$:
\begin{equation*}
  N_p(M) \coloneqq \Bigl\{\bigl(M\xi,\mathcal J_p(\xi)\bigr) \colon \xi \in \C^d,~\abs{\xi}=1 \Bigr\}.
\end{equation*}
Clearly, one could also define the $p$-numerical range for an operator $L$ on a complex Hilbert space $\Hi$ with the necessary changes.
In any case, sticking to $\Hi = \C^{d}$, in the above context, $\Delta_p(\mu) \geq 0$ means that $N_p(\mu(x)) \subseteq [\Re z \geq 0]=\Sigma_{\frac\pi2}$ for almost all $x \in \Omega$.
Of course, $N_2(M) = N(M)$, the usual numerical range.
Further, we have $N_p(M) = N_{p'}(M)$, due to $(M\xi,\cJ_p(\xi)) = (M(\ii \xi),\cJ_{p'}(\ii \xi))$ for all $\xi \in \C^d$.
Moreover,
\begin{equation}
  (M\xi,\xi) = \tfrac12 \bigl(M\xi,\mathcal J_p(\xi)\bigr)  + \tfrac12\bigl(M\xi,\mathcal J_{p'}(\xi)\bigr),\label{eq:p-numerical-range-convex-combo}
\end{equation}
so $N(M) = N_2(M) \subseteq \frac12 N_p(M) + \frac12 N_{p'}(M) = \frac12 N_p(M) + \frac12 N_p(M)$.
In particular, whenever $N_p(M)$ is contained in a convex cone---such as a sector---for some $p \in (1,\infty)$, then so is $N(M)$.

\begin{theorem}\label{thm:numerical-range-p-elliptic-operator}
  Let $p \in \cI(A)$ and suppose that $\theta \in [0,\frac\pi2]$ is such that $N_p(\mu(x)) \subseteq \Sigma_\theta$ for almost all $x \in \Omega$.
  Then $N(A_p) \subseteq \Sigma_\theta$.
\end{theorem}

We postpone the proof of Theorem~\ref{thm:numerical-range-p-elliptic-operator} and first complement it with an estimate for $N_p(\mu(x))$.
However, we already mention that the $\theta$ in Theorem~\ref{thm:numerical-range-p-elliptic-operator} necessarily satisfies $\theta \geq \omega(\mu)$, the optimal angle of accretivity for $\mu$.
This follows from~\eqref{eq:p-numerical-range-convex-combo} and its following comment.

For the estimate, we briefly introduce some notation, following~\cite[Section~5.3]{CarbonaroDragicevic2020}.
Put $\sigma_{p} \coloneqq \frac{p-2}{2\sqrt{p-1}}$, and also $\tau_p \coloneq \frac{p}{2\sqrt{p-1}}$.
For $M \in \cL(\C^d) \simeq \C^{d \times d}$, by writing $M = U + \ii V$ we mean that $\Re M = U$ and $\Im M = V$, each in $\R^{d\times d}$.
We further put $M_s \coloneq (M+M^T) /2$ and $M_a \coloneq (M-M^T) / 2$.
Given a matrix $M = U + \ii V$ with $U_s$ positive definite, we set
\begin{equation*}
  \cV_p(V) \coloneq \sigma_pV_s +\tau_p V_{a} \quad \text{and} \quad \cW_p(M) = \cW_p(U + \ii V) \coloneq U_s^{-\frac{1}{2}}\cV_p(V) U_s^{-\frac{1}{2}}.
\end{equation*}
Note that we slightly deviate from~\cite{CarbonaroDragicevic2020} by writing $\cV_p(V)$ instead of $\cV_p(M)$.
This will pay off later.
In any case, with $\cW_p(M)$ we can in fact characterize $p$-ellipticity:

\begin{lemma}
  \label{lem:the-new-estimate}
  Let $p \in (1,\infty)$ and put $c^\bullet \coloneqq \esssup_{x\in\Omega}\norm{\cW_p(\mu(x))}$.
  Then $\mu$ is $p$-elliptic if and only if $c^\bullet < 1$.
\end{lemma}

\begin{proof}
  In the proof of~\cite[Proposition~5.13]{CarbonaroDragicevic2020}, the authors show that given $M = U + \ii V \in \C^{d \times d}$ with $U_s$ positive definite, there holds $\Delta_p(M) \geq 2c$ (with $\Delta_p(M)$ as in~\eqref{eq:p-elliptic} for the constant function $M$) \emph{if and only if}
  \begin{equation}\label{eq:cd-charact}
    \abs{u}^2 + \abs{v}^2 - 2 \abs[\big]{\bigl(\cW_p(M)v,u\bigr)} \geq c \bigl(p \abs{U_s^{-\frac{1}{2}}v}^2 + p' \abs{U_s^{-\frac{1}{2}}u}^2\bigr) \qquad (u,v \in \R^{d}),
  \end{equation}
  and leave the remaining steps to the reader.
  Since the uniformity in $x$ enters only in these steps, we give them.
  There is a subset $\Lambda \subseteq \Omega$ with $\lvert \Omega\setminus\Lambda\rvert = 0$ such that for all $x \in \Lambda$, we have $\norm{\cW_p(\mu(x))} \leq c^\bullet$ and $\Delta_p(\mu(x)) \geq \Delta_p(\mu)$, also $U_s(x)$ is positive definite with $\norm{\mu}_{L^{\infty}(\Omega;\cL(\C^d))}^{-1}\abs{w}^2 \leq \abs{U_s(x)^{-1/2}w}^2 \leq m_\bullet^{-1}\abs{w}^2$ for $w \in \R^d$.
  In the following, consider $x \in \Lambda$.
  Suppose that $c^\bullet < 1$.
  Then the left-hand side in~\eqref{eq:cd-charact} for $M = \mu(x)$ can be estimated from below by $(1-c^{\bullet})(\abs{u}^2+\abs{v}^2)$, so~\eqref{eq:cd-charact} holds for $M = \mu(x)$ with $c = (1-c^\bullet)m_\bullet/(p \vee p') > 0$, and $\mu$ is $p$-elliptic.
  Vice versa, suppose that $\Delta_p(\mu) \eqcolon 2c > 0$.
  Then
  \begin{equation*}
    \abs{u}^2 + \abs{v}^2 - 2 \abs[\big]{\bigl(\cW_p(\mu(x))v,u\bigr)} \geq c \bigl(p \abs{U_s(x)^{-\frac{1}{2}}v}^2 + p' \abs{U_s(x)^{-\frac{1}{2}}u}^2\bigr) \qquad (u,v \in \R^{d}).
  \end{equation*}
  Choose $v \in \R^d$ with $\abs{v} = 1$ and $\abs{\cW_p(\mu(x))v} = \norm{\cW_p(\mu(x))}$, as well as $u = \frac{\cW_p(\mu(x))v}{\abs{\cW_p(\mu(x))v}}$.
  (If $\norm{\cW_p(\mu(x))} = 0$, choose arbitrary unit vectors $u,v \in \R^d$.)
  Then the left-hand side becomes $2 - 2 \norm{\cW_p(\mu(x))}$.
  Thus, estimating the right-hand side from below, we find
  \begin{equation*}
    2 - 2 \norm{\cW_p(\mu(x))} \geq c(p+p')\norm{\mu}_{L^{\infty}(\Omega;\cL(\C^d))}^{-1}.
  \end{equation*}
  So $c^\bullet \leq 1 - c(p+p')\norm{\mu}_{L^{\infty}(\Omega;\cL(\C^d))}^{-1}/2 < 1$.
\end{proof}

\begin{remark}
  \label{rem:suff-cond-p-ellipt}
  Lemma~\ref{lem:the-new-estimate} is a precise characterization but not trivial to check.
  An elementary sufficient condition for $p$-ellipticity is that there exists $\varepsilon > 0$ such that $\abs{\sigma_p}\,\norm{\Im\mu(x)}_{\R^d\to\R^d} \leq (1-\varepsilon)\,m(x)$ for almost all $x \in \Omega$.
  Then $\Delta_p(\mu) \geq 2\varepsilon m_\bullet/(p \vee p')$.
  This follows from the proof of~\cite[Lemma~2.16]{TerElstEtAl2021}, where the condition is stated with uniform data, and the constant cannot be improved in general, see~\cite[Example~2.17]{TerElstEtAl2021}.
  In terms of $\eta \coloneqq \esssup_{x\in\Omega}\norm{\Im\mu(x)}_{\R^d\to\R^d}/m(x)$, the condition holds for all $p$ with $\abs{\frac{1}{2} - \frac{1}{p}} < \abs{\frac{1}{2} - \frac{1}{q}}$, where $q \in (2,\infty)$ satisfies $\eta \sigma_q = 1$.
  (If $\eta = 0$, then $\mu$ is real and $p$-elliptic for every $p \in (1,\infty)$ anyway.)
\end{remark}

The next lemma is in the proof of~\cite[Proposition~5.22]{CarbonaroDragicevic2020}.

\begin{lemma}%
  \label{lem:carbonaro-dragicevic-inequality}
  Let $p \in (1,\infty)$ and $M = U + \ii V \in \C^{d\times d}$ with $U_s$ positive definite.
  Then
  \begin{equation*}
    \abs[\big]{\Im (V\xi,\cJ_p(\xi))} \leq \norm{\cW_p(M)} \Re (U\xi,\cJ_p(\xi)) \qquad (\xi \in \C^d).
  \end{equation*}
\end{lemma}

We next give the actual sector estimate.
It builds upon the following simple decomposition for $M = U + \ii V \in \C^{d\times d}$ and $\xi \in \C^d$:
\begin{equation}
  \label{eq:p-elliptic-decomp}
  \begin{split}
    \Re (M\xi,\cJ_p(\xi)) & = \Re (U\xi,\cJ_p(\xi)) - \Im (V\xi,\cJ_p(\xi)), \\
    \Im (M\xi,\cJ_p(\xi)) & = \Im (U\xi,\cJ_p(\xi)) + \Re (V\xi,\cJ_p(\xi)).
  \end{split}
\end{equation}
The plan is to estimate all four addends in~\eqref{eq:p-elliptic-decomp} in terms of $\Re (U\xi,\cJ_p(\xi))$ as suggested by Lemma~\ref{lem:carbonaro-dragicevic-inequality}.
It will thus also be useful to note that for any real matrix $X \in \R^{d\times d}$,
\begin{equation}
  \label{eq:p-elliptic-U-identity}
  \Re (X\xi,\cJ_p(\xi)) = \tfrac{2}{p'}(X_s\alpha,\alpha) + \tfrac{2}{p}(X_s\beta,\beta) \qquad (\xi = \alpha + \ii\beta \in \C^d).
\end{equation}

\begin{theorem}
  \label{thm:new-complex-angle-estimate}
  Let $p \in (1,\infty)$ and suppose that $\mu$ is $p$-elliptic.
  For almost every $x \in \Omega$, denote by $\theta_p(x)$ the angle from Theorem~\ref{thm:chill} for $\Re \mu(x)$, i.e., $\cos(\theta_p(x)) = \frac{2\sqrt{p-1}}{p} \cos(\omega(\Re \mu(x)))$.
  Then $N_p(\mu(x)) \subseteq \Sigma_{\alpha_p}$ for almost all $x \in \Omega$, with
  \begin{equation}
    \label{eq:angle-estimate}
    \tan(\alpha_p) \coloneq \esssup_{x\in\Omega}\left[\frac{\tan(\omega(\mu_{s}(x))) + \tan(\theta_p(x))}{1-\norm{\cW_p(\mu(x))}}\right].
  \end{equation}
\end{theorem}

\begin{proof}
  By Assumption~\ref{a-lowerbound} and Lemma~\ref{lem:the-new-estimate}, there is a subset $\Lambda \subseteq \Omega$ with $\lvert \Omega\setminus\Lambda\rvert = 0$ such that $U_s(x)$ is positive definite and $\norm{\cW_p(\mu(x))} \leq c^\bullet < 1$ for all $x \in \Lambda$.
  In view of the essential supremum in~\eqref{eq:angle-estimate}, it suffices to argue pointwise for $x \in \Lambda$.
  So let $x \in \Lambda$ and set $\mu(x) = U(x) + \ii V(x)$.
  From~\eqref{eq:p-elliptic-decomp} and Lemma~\ref{lem:carbonaro-dragicevic-inequality}, we have
  \begin{equation*}
    \Re (\mu(x)\xi,\cJ_p(\xi)) \geq \bigl(1-\norm{\cW_p(\mu(x))}\bigr) \Re (U(x)\xi,\cJ_p(\xi)) \qquad (\xi \in \C^{d}).
  \end{equation*}
  Now for the imaginary part.
  Set $M = U_s(x) + \ii U(x)$.
  Then $(\Re M)_s = U_s(x)$ is positive definite.
  Lemma~\ref{lem:carbonaro-dragicevic-inequality} applies and shows that
  \begin{equation*}
    \abs[\big]{\Im (U(x)\xi,\cJ_p(\xi))} \leq \norm{\cW_p(M)} \Re (U(x)\xi,\cJ_p(\xi)) \qquad (\xi \in \C^d),
  \end{equation*}
  where we have already used that in fact $ \Re (U_s(x)\xi,\cJ_p(\xi)) = \Re (U(x)\xi,\cJ_p(\xi))$ for all $\xi \in \C^d$ via~\eqref{eq:p-elliptic-U-identity}.
  Now
  \begin{equation*}
    \cW_p(M) = U_s(x)^{-1 /2}\cV_p(U(x))U_s(x)^{-1 /2} = \sigma_pI + \tau_pU_s(x)^{-\frac{1}{2}}U_a(x)U_s(x)^{-\frac{1}{2}},
  \end{equation*}
  and this is precisely the matrix at the heart of the matter in the proof of the optimal angle for real matrices as in Theorem~\ref{thm:chill}, see~\cite[Proof of Theorem~3.2]{ChillMeinlschmidtRehberg2021}.
  We recall the argument.
  Put $B \coloneqq U_s(x)^{-1 /2}U_a(x)U_s(x)^{-1 /2}$.
  Since $B$ is antisymmetric, $\norm{\cW_p(M)}^2 = \sigma_p^2 + \tau_p^2\norm{B}^2$.
  Further, $\ii B$ is Hermitian, and we have $(U_s(x)\xi,\xi) > 0$ and $(U_a(x)\xi,\xi) \in \ii\R$ for all $0 \neq \xi \in \C^d$.
  Hence
  \begin{equation*}
    \norm{B} = \max_{0 \neq \eta \in \C^d}\frac{\abs{(B\eta,\eta)}}{\abs{\eta}^2}
    = \max_{0 \neq \xi \in \C^d}\frac{\abs{(U_a(x)\xi,\xi)}}{(U_s(x)\xi,\xi)}
    = \tan(\omega(U(x))),
  \end{equation*}
  so
  \begin{equation*}
    \norm{\cW_p(M)} = \tau_p \sqrt{(1-\tfrac{2}{p})^2 + \tan(\omega(U(x)))^2} = \tan(\theta_p(x)),
  \end{equation*}
  by~\eqref{eq:lpWinkel}.
  Finally, we have, relying on~\eqref{eq:p-elliptic-U-identity} for all $\xi = \alpha + \ii\beta\in\C^d$,
  \begin{align*}
    \abs[\big]{\Re (V(x)\xi,\cJ_p(\xi))} & \leq \tfrac{2}{p'}\abs{(V_s(x)\alpha,\alpha)} + \tfrac{2}{p}\abs{(V_s(x)\beta,\beta)} \\
                                   & \leq \norm[\big]{U_{s}(x)^{-\frac{1}{2}}V_s(x)U_s(x)^{-\frac{1}{2}}} \cdot \bigl[\tfrac{2}{p'}\abs{U_s(x)^{\frac{1}{2}}\alpha}^{2} + \tfrac{2}{p}\abs{U_s(x)^{\frac{1}{2}}\beta}^2\bigr] \\ & = \norm[\big]{U_{s}(x)^{-\frac{1}{2}}V_s(x)U_s(x)^{-\frac{1}{2}}} \Re (U(x)\xi,\cJ_p(\xi)).
  \end{align*}
  Put $C = U_{s}(x)^{-1 /2}V_s(x)U_s(x)^{-1/ 2}$.
  Then $C$ is symmetric, so
  \begin{equation*}
    \norm{C} = \max_{0 \neq \eta \in \C^d} \frac{\abs{(C\eta,\eta)}}{\abs{\eta}^2} = \max_{0 \neq \xi \in \C^d} \frac{\abs{(V_s(x)\xi,\xi)}}{(U_s(x)\xi,\xi)} = \tan(\omega(\mu_s(x))).
  \end{equation*}
  Here $V_s(x)$ and $U_s(x)$ are real and symmetric matrices with $U_s(x)$ positive definite, so $(V_s(x)\xi,\xi) \in \R$ and, as before, $(U_s(x)\xi,\xi) > 0$ for all $0 \neq \xi \in \C^d$.
  Putting it all together along~\eqref{eq:p-elliptic-decomp}, we have
  \begin{equation*}
    \sup_{\abs{\xi} = 1}\frac{\abs{\Im (\mu(x)\xi,\cJ_p(\xi))}}{\Re (\mu(x)\xi,\cJ_p(\xi))} \leq \frac{\tan(\omega(\mu_{s}(x))) + \tan(\theta_p(x))}{1-\norm{\cW_p(\mu(x))}},
  \end{equation*}
  and this completes the proof.
\end{proof}

\begin{remark}
  \label{rem:comments-on-angle}
  The angle~\eqref{eq:angle-estimate} is new in general and has several appealing properties.
  It reproduces the angle in Theorem~\ref{thm:chill} if $\mu$ is in fact a real coefficient matrix.
  Corollary~\ref{cor:complex-operator} thus contains Theorem~\ref{thm:chill} as a special case.
  Further, if $\mu(x)$ is symmetric (but possibly complex), then $\mu(x) = \mu_s(x)$,
  \begin{equation*}
    \norm{ \cW_p(\mu(x))} = \abs{\sigma_{p}} \norm[\big]{U_s(x)^{-\frac{1}{2}}V_s(x) U_s(x)^{-\frac{1}{2}}} = \abs{\sigma_p} \tan(\omega(\mu(x)))
  \end{equation*}
  and $\tan(\theta_p(x)) = \abs{\sigma_p}$.
  Thus
  \begin{equation*}
    \frac{\tan(\omega(\mu(x))) + \tan(\theta_p(x))}{1-\norm{\cW_p(\mu(x))}} = \frac{\tan(\omega(\mu(x))) + \abs{\sigma_{p}}}{1-\abs{\sigma_p}\tan(\omega(\mu(x)))} = \tan\bigl(\omega(\mu(x)) + \arctan(\abs{\sigma_p})\bigr).
  \end{equation*}
  This is the angle derived in~\cite[Theorem~1.2]{Do2021} for symmetric coefficient matrices under regularity assumptions on $\mu$ and $\Omega$ and with Dirichlet boundary conditions.
  It is in fact the optimal angle for $N_p(\mu(x))$ in this setting, as already observed in~\cite[Remark~1.3]{Do2021}: we have $N_p(\mu(x)) \subseteq \Sigma_\theta$ if and only if both matrices $\mu_{\pm\theta}(x) \coloneq e^{\pm\ii(\frac\pi2-\theta)}\mu(x)$ satisfy $\Re (\mu_{\pm\theta}(x)\xi,\cJ_p(\xi)) \ge 0$ for all $\xi \in \C^d$.
  For symmetric matrices, the latter is characterized in~\cite[Proposition~5.18]{CarbonaroDragicevic2020}, which is the condition of Cialdea and Maz'ya~\cite{CialdeaMazya2005}, and it holds if and only if $\theta \geq \omega(\mu(x)) + \arctan(\abs{\sigma_p})$.
  Moreover, as for real coefficients, see the comment after Theorem~\ref{thm:chill}, antisymmetric parts of $\mu(x)$ can only increase the bound.
  This follows from the real case and $\norm{\cW_p(\mu(x))} \geq \abs{\sigma_p}\tan(\omega(\mu_s(x))) = \norm{\cW_p(\mu_s(x))}$, with equality if $V_a(x) = 0$.
  (For the inequality, test with an eigenvector of $U_s(x)^{-1/2}V_s(x)U_s(x)^{-1/2}$ for the eigenvalue of largest modulus.)
  Finally, we mention that the works of Do~\cite{Do2021,Do2023} also treat degenerate coefficient matrices with $\essinf_{x\in\Omega} m(x) = 0$.
  It is thus interesting to note that the angle~\eqref{eq:angle-estimate} is invariant under scaling of $\mu$ by positive scalar functions.
  In particular it does not deteriorate for, say, $\mu(x) = \rho(x) \mu_0$ with a fixed coefficient matrix $\mu_0 \in \cL(\C^d)$ and $\rho(x) \downarrow 0$.
  This suggests that also the degenerate case could be dealt with, combining techniques of~\cite{Do2021,Do2023} with the above.
\end{remark}

\begin{corollary}\label{cor:complex-operator}
  Let $p \in (1,\infty)$ and suppose that $\mu$ is $p$-elliptic.
  Then $A_p$ is $m$-$\alpha_p$-accretive with $\alpha_p$ as in~\eqref{eq:angle-estimate}.
  In particular, the semigroup $T_p$ is holomorphic of angle $\sfrac\pi2-\alpha_p$, and the spectral angle $\phi(A_p)$ is not larger than $\alpha_p$.
\end{corollary}
\begin{proof}
  By Proposition~\ref{prop:p-elliptic-extrapolate-general}, $p \in \cI(A)$, and $A_p$ is $m$-accretive, see above.
  Theorems~\ref{thm:new-complex-angle-estimate} and~\ref{thm:numerical-range-p-elliptic-operator} give $N(A_p) \subseteq \Sigma_{\alpha_p}$.
  Moreover, $\alpha_p < \sfrac\pi2$, since the denominator in~\eqref{eq:angle-estimate} is at least $1-c^\bullet$ by Lemma~\ref{lem:the-new-estimate}, and the numerator is bounded since $\mu$ is bounded and uniformly elliptic.
  The rest follows from Proposition~\ref{prop:sector-resolvent-general}.
\end{proof}

We now prove Theorem~\ref{thm:numerical-range-p-elliptic-operator} which ties an angle estimate for the $p$-numerical range of the coefficient matrix to the operator.

\textbf{Proof of Theorem~\ref{thm:numerical-range-p-elliptic-operator}.}
We proceed in several steps with the big picture being as follows: Due to the assumption $p \in \cI(A)$, we have the semigroups $T_p$ and associated (negative) generators $A_p$ on $L^p(\Omega)$ at hand.
With the single-valued duality mapping $J(u) = u^* = \abs{u}^{p-2}\overline u/\norm{u}_{L^p(\Omega)}^{p-1}$ on the strictly convex Banach space $L^p(\Omega)$, we would be interested to write, for all $u \in \dom(A_p)$ with $\norm{u}_{L^p(\Omega)} = 1$ to get rid of the normalizing factor,
\begin{equation*}
  u^*(A_p u) \approx \int_\Omega \bigl(A_{p}u\bigr) \abs{u}^{p-2}\overline u = \ft_V(u,\abs{u}^{p-2}u)
\end{equation*}
and then estimate the form, which will lead to the $p$-ellipticity expression.
However, it is in general not clear whether $\abs{u}^{p-2}u \in V$, the form domain of $\ft_V$.
We thus go along a few detours, essentially working backwards.
In \textbf{Step 1}, we show that (figuratively)
\begin{equation*}
  \ft(u,\abs{u}^{p-2}u) = \int_\Omega \bigl(\mu(\alpha + \ii \beta),\cJ_p(\alpha + \ii \beta)\bigr)
\end{equation*}
for some functions $\alpha,\beta$, if $u$ is smooth and bounded from above and away from zero.
This computation is essentially due to Cialdea and Maz'ya~\cite{CialdeaMazya2005}.
Together with the assumption on $N_p(\mu(x))$, it gives $\ft(u,\abs{u}^{p-2}u) \in \Sigma_\theta$.
In \textbf{Step 2}, we pass to general $u \in W^{1,2}(\Omega)$ by means of a cut-off $\abs{\cdot}_K$ at levels $1/K$ and $K$, and show that $\ft(u,\abs{u}_K^{p-2}u) \in \Sigma_\theta$.
In \textbf{Step 3}, we let $K \to \infty$ to obtain $N(A_p) \subseteq \Sigma_\theta$.

So, from now on, suppose that $\theta \in [0,\frac\pi2]$ is such that $N_p(\mu(x)) \subseteq \Sigma_\theta$ for almost all $x \in \Omega$.
Here is \textbf{Step 1}, in sufficient generality for all purposes later:
\begin{lemma}\label{lem:dissipa}
  Let $\Lambda \subseteq \Omega$ be open and let $u \in C^1(\Lambda) \cap W^{1,2}(\Lambda)$.
  If $p \in (1,2)$, suppose further that $\abs{u}$ is bounded below by a positive constant.
  Then
  \begin{equation*}
    \int_\Lambda \bigl(\mu\nabla u,\nabla (\abs{u}^{p-2} u )\bigr) \in \Sigma_\theta.
  \end{equation*}
\end{lemma}

\begin{proof}
  Observe that due to the supposed properties of $u$, all the subsequent formal manipulations, essentially following Cialdea/Maz'ya~\cite{CialdeaMazya2005}, are strictly justified.
  Indeed, if $p > 2$, then $u(x) = 0$ implies that $\nabla(\abs{u}^{p-2}u)(x) = 0$, cf.~\cite[Proof of Theorem~3.2]{ChillMeinlschmidtRehberg2021}, so we can without loss of generality replace $\Lambda$ by $\Lambda \cap [u \neq 0]$.
  For $p<2$, the supposed lower bound on $\abs{u}$ prevents a singularity and differentiability issues in $\abs{u}^{p-2}$.
  In any case, putting $v\coloneqq |u|^\frac {p-2}{2}u$, we calculate
  \begin{align*}
    \int\limits_\Lambda \bigl(  \mu \nabla u,
    \nabla ( \abs{u}^{p-2} {u} )\bigr)   &=
                                      \int\limits_\Lambda \Bigl(  \mu \nabla \bigl (|v|^\frac {2-p}{p}v \bigr ),  \nabla \bigl (
                                      |v|^\frac {p-2}{p} v\bigr )\Bigr)  \\
                                    &= \int\limits_\Lambda (\mu \nabla v,\nabla v)
                                      -\bigl(1-\tfrac {2}{p}\bigr )^2 \int\limits_\Lambda (\mu \nabla |v|,\nabla |v|) \\
                                    & \qquad \qquad +\bigl (1-\tfrac {2}{p} \bigr ) \Biggl( \int\limits_\Lambda \Bigl(  \mu \frac {\overline v}{|v|}
                                      \nabla v ,  \nabla | v|\Bigr)
                                      - \int\limits_\Lambda \Bigl(  \mu \nabla |v|,  \frac {\overline v}{ |v|}  \nabla v
                                      \Bigr)   \Biggr).%
  \end{align*}
  Further following~\cite[Corollary~1]{CialdeaMazya2005}, we put $\phi\coloneqq \Re \bigl (\frac {\overline v}{|v|}\nabla v \bigr )$ and $\psi\coloneqq \Im \bigl (\frac {\overline v}{|v|}\nabla v \bigr )$.
  So:
  \begin{equation*}%
    \nabla |v| = \nabla \bigl (v \overline v \bigr )^{\frac {1}{2}}= \frac {1}{2} \frac {1}{|v|}
    \bigl(v \nabla \overline v + \overline v \nabla v \bigr) = \Re \Bigl( \frac {\overline v}{|v|}
    \nabla v\Bigr ) = \phi.
  \end{equation*}
  It follows that
  \begin{equation*}%
    \int\limits_\Lambda (\mu \nabla v,\nabla v)  =
    \int\limits_\Lambda \Bigl(  \mu \frac {\overline v}{|v|}\nabla v,
    \frac {\overline v}{|v|}  \nabla  v \Bigr)
    = \int\limits_\Lambda \bigl(  \mu(\phi +\ii  \psi), \phi +\ii  \psi \bigr)
  \end{equation*}
  and overall
  \begin{align*}
    \int\limits_\Lambda \bigl(  \mu \nabla u,
    \nabla ( \abs{u}^{p-2} {u} )\bigr)
    &=\int\limits_\Lambda \bigl(  \mu (\phi+\ii  \psi),
      \phi+\ii  \psi \bigr)
      - \bigl(1 -\tfrac {2}{p}\bigr)^2 \int\limits_\Lambda (  \mu \phi,  \phi )  \nonumber \\
    & \qquad \qquad + \bigl (1-\tfrac {2}{p} \bigr ) \Bigl( \int\limits_\Lambda \bigl(  \mu (\phi +\ii  \psi)
      ,  \phi \bigr)
      -\int\limits_\Lambda \bigl (  \mu\phi , \phi+\ii \psi \bigr)  \Bigr ) \nonumber \\
    &= \Bigl ( 1 -\bigl ( 1- \tfrac {2}{p}\bigr )^2 \Bigl ) \int\limits_\Lambda (  \mu \phi,
      \phi )   + \int\limits_\Lambda (  \mu \psi,\psi  )
    \\
    & \qquad \qquad
      + \ii\, \frac {2}{p'} \int\limits_\Lambda (  \mu \psi,  \phi )  -
      \ii \, \frac {2}{p} \int\limits_\Lambda (  \mu \phi,  \psi )
  \end{align*}
  Staring at the last right-hand side for a minute and setting $\varphi \coloneqq \frac{2\phi}p$, we find
  \begin{equation*}
    \int\limits_\Lambda ( \mu \nabla u,
    \nabla ( \abs{u}^{p-2} {u} ))  = \frac{p}2\int_\Lambda \bigl(\mu(\varphi + \ii \psi),\cJ_p(\varphi + \ii \psi)\bigr)
  \end{equation*}
  and the claim follows.
\end{proof}

We proceed with \textbf{Step 2} to remove the restriction on $u$ in Lemma~\ref{lem:dissipa}.
It will cost us a cut-off, but one that we can deal with later.
Depending on whether $p$ is larger or less than $2$, the problematic region to be cut off changes.

\begin{lemma}\label{lem:Dissipa}
  Let $u \in C^1(\Omega) \cap W^{1,2}(\Omega)$.
  Define, for $K >0$ and $z \in \C$,
  \[ |z|_K^{p-2} \coloneqq \begin{cases} \bigl(|z| \wedge K\bigr)^{p-2} & \text{if}~p\in [2,\infty), \\[0.5em]
    \bigl(\lvert z \rvert \vee \frac {1}{K}\bigr)^{p-2} & \text{if}~p \in (1,2).
  \end{cases}
\]
Then
\begin{equation*}
  \int_{\Omega} \bigl( \mu \nabla u, \nabla (|u|_K^{p-2}u) \bigr)\in \Sigma_\theta.
\end{equation*}
\end{lemma}
\begin{proof}
  We argue for $p \geq 2$.
  The case $p < 2$ works in the same way.
  We have by the usual (distributional) rules of differentiation
  \begin{equation}\label{eq:distributrule}
    \frac {\partial }{\partial x_j} \bigl( |u|_K^{p-2}u \bigr) = \begin{cases} \frac {\partial }{\partial x_j} \bigl( |u|^{p-2}u \bigr) & \text{if }
                                                                                                              |u| < K, \\[0.5em]
      K^{p-2} \frac{\partial u}{\partial x_j} & \text{if } |u| \ge K.    \end{cases}
  \end{equation}
  This gives
  \begin{equation}
    \int_{\Omega} \bigl( \mu \nabla u, \nabla (|u|_K^{p-2}u) \bigr) =\int_{[\lvert u| < K]} \bigl( \mu \nabla u, \nabla
    (|u|^{p-2}u) \bigr) + K ^{p-2} \int_{[|u| \ge K]} ( \mu \nabla u, \nabla u )    \label{eq:addends}
  \end{equation}
  We put $\Lambda \coloneqq [|u| < K]$, which is an open subset of $\Omega$.
  By the foregoing Lemma~\ref{lem:dissipa}, for the first term on the right of~\eqref{eq:addends} we have
  \[
    \int_{\Lambda} ( \mu \nabla u, \nabla (|u|^{p-2}u) ) \in \Sigma_\theta
  \]
  The second term is contained in $\Sigma_{\omega(\mu)} \subseteq \Sigma_\theta$, recall~\eqref{eq:p-numerical-range-convex-combo}.
  With $\Sigma_\theta$ being a convex cone, it follows that
  \[
    \int_{\Omega} ( \mu \nabla u, \nabla (|u|_K^{p-2}u) ) \in \Sigma_\theta.\qedhere
  \]
\end{proof}

\begin{lemma}\label{lem:cutoff}
  Let $u \in W^{1,2}(\Omega)$ and $K>0$.
  Then $\abs{u}_K^{p-2}u \in W^{1,2}(\Omega)$ and $\ft(u,\abs{u}_K^{p-2}u) \in \Sigma_\theta$.
\end{lemma}
\begin{proof} Let $K >0$ be arbitrary but fixed.
  We note that the chain rule~\eqref{eq:distributrule} remains valid for $u \in W^{1,2}(\Omega)$.
  Thus, using that
  \begin{equation*}%
    \frac {\partial }{\partial x_j} \abs{v} = \frac {1}{2} \frac {\overline v}{|v|} \frac {\partial v }{\partial x_j} + \frac {1}{2} \frac { v}{|v|} \frac {\partial \overline v }{\partial x_j} = \frac{\Re\bigl(v\,\overline{\frac{\partial v}{\partial x_j}}\bigr)}{\abs{v}},
  \end{equation*}
  for $v \in W^{1,2}(\Omega)$~(\cite[Proposition~4.4]{Ouhabaz2005}), we find that for $p \geq 2$,
  \begin{equation}\label{eq:diff-duality}
    \frac {\partial }{\partial x_j} \bigl(|u|_K^{p-2}u \bigr) = |u|_K^{p-2} \frac {\partial u }{\partial x_j}
    + \chi_{[\abs{u} < K]}(p-2) u |u|^{p-4} \Re\Bigl(u\,\overline{\frac{\partial u}{\partial x_j}}\Bigr).
  \end{equation}
  Similarly, for $p \leq 2$,
  \begin{equation}\label{eq:diff-duality-other}
    \frac {\partial }{\partial x_j} \bigl(|u|_K^{p-2}u \bigr) = |u|_K^{p-2} \frac {\partial u }{\partial x_j}
    + \chi_{[\frac{1}{K}< \abs{u}]}(p-2) u |u|^{p-4} \Re\Bigl(u\,\overline{\frac{\partial u}{\partial x_j}}\Bigr).
  \end{equation}
  It is now easily seen that $\abs{u}_K^{p-2}u \in W^{1,2}(\Omega)$ for every $u \in W^{1,2}(\Omega)$ in both cases.
  This is the first assertion.
  In order to remove the smoothness assumption from Lemma~\ref{lem:Dissipa}, let $u \in W^{1,2}(\Omega)$ be arbitrary.
  Then, by the Meyers-Serrin theorem, there is a sequence $(u_n)$ in $C^1(\Omega) \cap W^{1,2}(\Omega)$ such that $u_n \to u$ in $W^{1,2}(\Omega)$.
  From~\eqref{eq:diff-duality} and~\eqref{eq:diff-duality-other}, we see that the sequence $(\abs{u_n}_K^{p-2}u_n)$ is uniformly bounded in $W^{1,2}(\Omega)$ for both cases of $p$, and thus admits a weakly convergent subsequence that converges to some $v \in W^{1,2}(\Omega)$.
  Passing to another subsequence, we also know that $(u_n)$ converges pointwise almost everywhere to $u$.
  It follows that $v = \abs{u}_K^{p-2} u$ and so, from Lemma~\ref{lem:Dissipa},
  \begin{equation*}
    \ft(u,\abs{u}_K^{p-2}u)  = \lim_{n\to\infty}\ft(u_n,\abs{u_n}_K^{p-2}u_n) \in \Sigma_\theta.\qedhere
  \end{equation*}
\end{proof}

We collect the overall conclusion in \textbf{Step 3}:

\begin{proof}[Proof of Theorem~\ref{thm:numerical-range-p-elliptic-operator}]
  Let $u \in \dom(A_p)$ with $\norm{u}_{L^p(\Omega)} = 1$.
  We show that $u^*(A_pu) \in \Sigma_\theta$.
  In fact, it will be enough to show the latter for $u$ from a core of $\dom(A_p)$, because if $v \in \dom(A_p)$ with $\norm{v}_{L^p(\Omega)}=1$ is approximated by a sequence $(u_n)$ in a core of $\dom(A_p)$, then $A_p u_n \to A_pv$ and $u_n^* \coloneqq \abs{u_n}^{p-2}\overline{u_n}/\norm{u_n}_{L^p(\Omega)}^{p-1}$ converges weakly to $\abs{v}^{p-2}\overline{v}= v^*$ in $L^p(\Omega)^*$.
  Thus, $u_n^*(A_p u_n) \to v^*(A_p v)$ and if $u_n^*(A_p u_n) \in \Sigma_\theta$, then so is $v^*(A_p v)$.
  Since the semigroups $T_2$ and $T_p$ are consistent by construction, $\dom(A_p) \cap \dom(A_2)$ is a core for $\dom(A_p)$, see~\cite[Lemma~3.1]{ChillMeinlschmidtRehberg2021}, and that one will work for us.
  So we show the assertion for $u \in \dom(A_p) \cap \dom(A_2)$ with $\norm{u}_{L^p(\Omega)} = 1$.
  Since we now have $u \in L^p(\Omega)$ at our disposal, we find that
  \begin{equation*}
    \abs[\big]{\abs{u}_K^{p-2}\overline{u}} \leq \abs{u}^{p-1} \in L^{p'}(\Omega).
  \end{equation*}
  Thus, with $\abs{u}_K^{p-2}\overline{u} \in L^{p'}(\Omega) \cap L^2(\Omega)$ and Dominated Convergence we find
  \begin{equation}\label{eq:Lebesgdomin}
    u^*(A_p u) = \int_\Omega (A_p u)  |u |^{p-2}\overline u = \lim _{K \to \infty} \int_\Omega (A_p u) |u |_K^{p-2} \overline u = \lim _{K \to \infty} \int_\Omega (A_2 u)  |u |_K^{p-2}\overline u.
  \end{equation}
  In order to relate the last integral in~\eqref{eq:Lebesgdomin} with $\ft_V(u,\abs{u}_K^{p-2}u)$, we need to show that $\abs{u}_K^{p-2}u \in V$ for any $K>0$.
  We already have $u \in \dom(A_2) \subseteq V$.
  \begin{itemize}
  \item Let $V$ be defined as the closure of $\underline{W}^{1,2}_D(\Omega)$ as in~\eqref{eq:closureII}.
    If $D = \emptyset$, then $V = W^{1,2}(\Omega)$, and by Lemma~\ref{lem:cutoff} $\abs{u}_K^{p-2}u \in W^{1,2}(\Omega)$, so we are done.
    Otherwise, there is a sequence $(u_n) \subseteq \underline{W}^{1,2}_D(\Omega)$ such that $u_n \to u$ in $W^{1,2}(\Omega)$.
    By definition, $\dist(\supp (u_n),D) > 0$ for each $n$, so we recover the same for each $\abs{u_n}_K^{p-2}u_n$.
    In particular, $(\abs{u_n}_K^{p-2}u_n) \subseteq \underline{W}^{1,2}_D(\Omega) \subseteq V$.
    As in the proof of Lemma~\ref{lem:cutoff}, there is a subsequence of $(\abs{u_n}_K^{p-2}u_n)$ that converges weakly to $\abs{u}_K^{p-2}u$ in $W^{1,2}(\Omega)$, and that implies that the latter is an element of $V$.
  \item If $V$ is defined via~\eqref{eq:subspace1}, then there is a sequence $(u_n) \subseteq C_D^\infty(\Omega)$ such that $u_n \to u$ in the $W^{1,2}(\Omega)$ norm.
    Again, there is a subsequence of $(\abs{u_n}_K^{p-2}u_n)$ that converges weakly to $\abs{u}_K^{p-2}u$ in $W^{1,2}(\Omega)$.
    It is thus enough to show that $(\abs{u_n}_K^{p-2}u_n) \subseteq V$.
    For this purpose, let $n$ be fixed.
    As before, it is clear that since $u_n \in C_D^\infty(\Omega)$, we still have $\dist(\supp (\abs{u_n}_K^{p-2}u_n), D) > 0$.
    From the definition of $C_D^\infty(\Omega)$, there is a $v \in C_c^\infty(\R^d)$ such that $u_n = v\rvert_\Omega$.
    Let $(\xi_k)$ be a usual mollifier family on $\R^d$.
    Let $v_k \coloneqq \abs{v}_K^{p-2}v * \xi_k \in C_c^\infty(\R^d)$.
    Then, for $k$ large enough, we retain $\dist(\supp(v_k),D) > 0$ and $v_k\rvert_\Omega \to \abs{u_n}_K^{p-2}u_n$ in $W^{1,2}(\Omega)$ as $k\to \infty$.
    But then the latter is an element of $V$ as desired.
  \end{itemize}
  Now, with the knowledge that $\abs{u}_K^{p-2}u \in V$, we can continue in~\eqref{eq:Lebesgdomin}, relating the last integral with $\ft_V$, to obtain from Lemma~\ref{lem:Dissipa}
  \begin{equation*}
    u^*(A_p u) = \lim_{K \to \infty} \ft_V(u,\abs{u}_K^{p-2}u) \in \Sigma_{\theta}.\qedhere
  \end{equation*}
\end{proof}

\section{More assumptions on geometry}\label{sec:more-assumpt-geom}

So far, the angles for the realizations $A_p$ in Theorems~\ref{thm:chill},~\ref{thm:functkalc} and Corollary~\ref{cor:complex-operator} were only perturbations or functions of $\omega(\ft_V)$ and $\omega(\mu)$ instead of $\omega(\ft_V)$ itself, but we have not required boundedness of $\Omega$, or \emph{any} kind of boundary regularity.
Investing just a tiny bit in geometry, known results allow to deduce that the $\cH^{\infty}$-angle of $A_p$ is at most $\omega(\ft_V)$.
We collect them here, also in preparation for Section~\ref{sec:oper-negat-sobol}.
The following assumption relates to the Neumann boundary part $N \coloneq \partial\Omega \setminus D$ of $\Omega$.

\begin{assumption}[Geometry of $N$]%
  \label{ass:geometrie-locally-uniform}
  Suppose that there are $\eps \in (0,1]$ and $\delta \in (0,\infty]$ for which the following is true, where we put $N_{\delta} \coloneq \{ x \in \R^{d} \colon \dist(x,N) < \delta\}$.
  \begin{enumerate}[(i)]
  \item For all $x,y \in \Omega \cap N_{\delta}$ with $\abs{x-y} < \delta$, there exists a rectifiable curve $\gamma \subseteq \Omega$ with length $\ell(\gamma) \leq \eps^{-1}\abs{x-y}$ and
    \begin{equation*}
      \dist(z,\partial\Omega \cap N_{\delta}) \geq \frac{{\eps \abs{z-x} \abs{z-y}}}{\abs{x-y}} \qquad (z \in \gamma).
    \end{equation*}
  \item There exists $\lambda >0$ such that for all connected components $\Lambda$ of $\Omega$ with $\partial\Lambda \cap N \neq \emptyset$, we have $\diam(\Lambda) \geq \lambda\delta$.
  \end{enumerate}
\end{assumption}

Assumption~\ref{ass:geometrie-locally-uniform} requires that $\Omega$ is a locally uniform domain, or an $\eps$-$\delta$-domain, near the boundary part $N$.
There is no regularity assumption on the boundary part $D$.
The assumption implies that there is a continuous linear extension operator for Sobolev spaces subject to mixed boundary conditions~\cite{BechtelEtAl2026}.
We will not need the extension operator explicitly, but much more so the following ingenious insight, going back to~\cite{Auscher2007} and elaborated recently in the present context in~\cite{Egert2018,Bechtel2024a,Bechtel2024b}.

\begin{proposition}[{\cite[Proposition~3.6]{Bechtel2024a}}]%
  \label{prop:geometry-h-infty}%
  Suppose that Assumption~\ref{ass:geometrie-locally-uniform} holds true.
  Let $p \in \cI(A)$ and denote by $-A_p$ the generator of $T$ in $L^p(\Omega)$.
  Then for any $\delta > 0$, the operator $A_p+\delta$ admits a bounded $\cH^\infty$-calculus on $L^p(\Omega)$ with $\phi_{\infty}(A_{p}+\delta)\leq\omega(\ft_V)$.
  If $\ft_V$ is coercive, then the foregoing also holds true for $\delta=0$.
\end{proposition}

We emphasize that Proposition~\ref{prop:geometry-h-infty} holds true equally well for complex and real coefficient functions.
So, under the given, extremely general geometric framework of Assumption~\ref{ass:geometrie-locally-uniform}, one \lq{}only\rq{} needs to understand $\cI(A)$.
By definition, $\cI(A)$ is an open interval; it contains all $p$ for which $\mu$ is $p$-elliptic, and $\cI(A) = (1,\infty)$ for real $\mu$ (Propositions~\ref{p-basicsform} and~\ref{prop:p-elliptic-extrapolate-general}).
For complex coefficients, $\cI(A)$ contains all $p$ such that $\lvert \frac1p-\frac12\rvert < \frac1d$, so $\cI(A) = (1,\infty)$ for $d=2$, see~\cite[Corollary~3.5]{Bechtel2024a}, and this condition is sharp in the sense that no $p$ with $\lvert \frac12-\frac1p \rvert > \frac1d$ is in $\cI(A)$ for \emph{every} complex coefficient function~\cite[p.~1284]{Egert2018}.
Under the additional Assumption~\ref{ass:geometric-d-1-set}, the limit cases $p = 2^* = \frac{2d}{d-2}$ and $p = 2_* = \frac{2d}{d+2}$ for $d \geq 3$ belong to $\cI(A)$, see~\cite[Remark~1.3]{Bechtel2024a},~\cite[Theorem~1.6]{Egert2018} and also~\cite{TerElstEtAl2021}.

\section{The operator in negative Sobolev scales}\label{sec:oper-negat-sobol}

In the previous results we have compiled many results of sector inclusions for the numerical range of second-order elliptic differential operators with real or complex coefficient functions on the function spaces $L^p(\Omega)$, for $1 < p < \infty$.
Yet, nowadays it is understood that for the treatment of many real world problems, spaces from the Lebesgue scale are not always sufficient for a satisfactory analytical problem setup, but rather larger function spaces should be considered.
This concerns in particular nonlinear, say, quasilinear problems via permanence principles for the dependence of the differential operator on the solutions, but also structural features of the problem setting.
For example, inhomogeneous Neumann boundary data cannot be incorporated in an $L^p(\Omega)$ problem formulation directly, whereas it can be represented easily in a weak formulation in dual spaces of Sobolev spaces such as $W^{-1,2}_D(\Omega)$, the antidual of $W^{1,2}_D(\Omega)$.
We define the latter to be the $V$ given by the closure of $C_D^\infty(\Omega)$ as in~\eqref{eq:subspace1}, and do not consider $V$ defined via~\eqref{eq:closureII} from now on.
Yet also $W^{-1,2}_D(\Omega)$, while a Hilbert space and thus quite convenient, is not always adequate to provide a suitable functional-analytical framework to analyze the problem in.
For example, the three-dimensional thermistor problem was shown to be wellposed in the ambient (spatial) function space $W^{-1,q}_D(\Omega)$ with $q>d=3$ (see~\cite{MeinlschmidtMeyerRehberg2017a,MeinlschmidtMeyerRehberg2017b}), which we define as follows:

\begin{definition}\label{def:negative-sobolev}
  Let $q \in (1,\infty)$.
  Define $W^{1,q}_D(\Omega)$ to be the completion of $C_D^\infty(\Omega)$ as in~\eqref{eq:subspace1} with respect to the $W^{1,q}(\Omega)$-norm.
  Put further $W^{-1,q}_D(\Omega) \coloneqq W^{1,q'}_D(\Omega)^*$, the \emph{anti}dual of $W^{1,q'}_D(\Omega)$.
\end{definition}

We will content ourselves with $q \geq 2$ in the following.
In that case, it is clear that if $\abs{\Omega} < \infty$, then $W^{-1,q}_D(\Omega) \hookrightarrow W^{-1,2}_D(\Omega)$.
In order to make use of the good results for $L^p(\Omega)$, we propose to use square root operators to transfer the $L^p(\Omega)$ results to the $W^{-1,q}_D(\Omega)$ scale.
Let us make this precise with a few definitions and notations, starting with \emph{positive} operators in the sense of Triebel~\cite[Section~1.14.1]{Triebel1978}:
\begin{definition}\label{def:positive-operator}
  We say that a densely defined closed linear operator $B$ on a Banach space $X$ is \emph{positive} if $(-\infty,0] \subset \rho(B)$ and there is a constant $C \geq 0$ such that
  \begin{equation}
    \label{eq:positive-operator}
    \norm[\big]{(B+t)^{-1}}_{X\to X} \leq \frac{C}{1+t} \qquad (t \in [0,\infty)).
  \end{equation}
\end{definition}
Note that a positive operator in the foregoing sense is sectorial as in Definition~\ref{d-Winkelsec}; the main difference is found in the requirement that $0 \in \rho(B)$ and the thus-adapted growth condition~\eqref{eq:positive-operator}.
A useful feature of positive operators is that one can define their negative fractional powers $B^{-\beta}$ by direct means---the Balakrishnan formula---without passing through a full functional calculus.
See~\cite[Section~3.8]{ArendtEtAl2001}, in particular~(3.51).

As before, we denote by $A$ the operator on $L^2(\Omega)$ induced by the form $\ft_V$ with $V = W^{1,2}_D(\Omega)$, corresponding to~\eqref{eq:subspace1}, and $A_p$ for $p \in \cI(A)$ as in Section~\ref{sec:operator-lpomega}.
Then $A_p+1$ is a positive operator and thus admits a unique (inverse) square root $(A_p+1)^{-1 / 2}$ on $L^p(\Omega)$.
We would like to have this inverse square root be an isomorphism $L^p(\Omega) \to W^{1,p}_D(\Omega)$, the so-called Kato square root property, and then use a duality argument to obtain an isomorphism $W^{-1,q}_D(\Omega) \to L^q(\Omega) $ for $q = p'$.
For this purpose, we add another minimal geometric assumption:

\begin{assumption}[Geometry of $D$]%
  \label{ass:geometric-d-1-set}%
  Adopt Assumption~\ref{ass:geometrie-locally-uniform} and suppose in addition that $D$ is \emph{Ahlfors-David-regular}, that is, there is a constant $C>0$ such that, with the $(d-1)$-dimensional Hausdorff measure $\cHaus_{d-1}$,
  \begin{equation*}
    Cr^{d-1} \leq \cHaus_{d-1}(D \cap B_r(x)) \leq C^{-1}r^{d-1} \qquad (x\in D,~r \leq \diam(D)).
  \end{equation*}
\end{assumption}

The square root property is strong and not easy to establish.
Fortunately, we are standing on the shoulders of giants, and with Assumption~\ref{ass:geometric-d-1-set} we obtain from the main results in~\cite{Bechtel2024a,Egert2018} (see also~\cite{Bechtel2024b}), applied to the adjoint coefficient function $\mu^*$ and the shifted operator, and taking adjoints as in~\cite[Theorem~5.1]{AuscherEtAl2015}:

\begin{proposition}\label{p-squareroot}
  Suppose Assumption~\ref{ass:geometric-d-1-set} and let $q \in \cI(A) \cap [2,\infty)$.
  Then there is an extension $E_{q}$ of $(A_q + 1)^{-1 /2} \colon L^q(\Omega) \to L^q(\Omega)$ such that \begin{equation}\label{eq:adjoint-square-root-kato} E_{q} \colon W^{-1,q}_D(\Omega) \to L^q(\Omega) \quad \text{is a topological isomorphism.}
  \end{equation}
\end{proposition}
Recall from the discussion regarding $\cI(A)$ in Section~\ref{sec:more-assumpt-geom} that if the coefficient function $\mu$ is real or if $d=2$, then we have~\eqref{eq:adjoint-square-root-kato} for all $q \in [2,\infty)$, whereas for a complex coefficient function with $d \geq 3$, it is at least always true for all $q \in [2,2^*]$, with the Sobolev conjugate $2^*$ of $2$.
In particular, the important case $q>d$ is already admissible for $d=3$ without further ado due to $2^* = 6$ in that case.

Let $\cA \colon W^{1,2}_D(\Omega) \to W^{-1,2}_D(\Omega)$ be the operator associated to $\ft_V$ on $V = W^{1,2}_D(\Omega)$, that is, such that
\begin{equation*}
  \ft_V(u,v) = \langle \cA u,v\rangle \qquad (u,v \in W^{1,2}_D(\Omega)).
\end{equation*}
We use~\eqref{eq:adjoint-square-root-kato} to transfer $A_q$ and its various good properties to $W^{-1,q}_D(\Omega)$.
The crucial point is that the isomorphism $E_q$ is adapted to $A_q$.

\begin{theorem}%
  \label{thm:w-1-q-results}
  Suppose Assumption~\ref{ass:geometric-d-1-set} and adopt the notation of Proposition~\ref{p-squareroot}.
  For $q \in \cI(A) \cap [2,\infty)$, define
  \begin{equation}
    \label{eq:w-1-q-operator}
    \cA_{q} \coloneq E_{q}^{-1} A_{q} E_{q},  \qquad \dom(\cA_q) \coloneq \Bigl\{ u \in W^{-1,q}_D(\Omega) \colon E_qu \in \dom(A_q)\Bigr\}.
  \end{equation}
  Then we have the following:
  \begin{enumerate}[(i)]
  \item $\cA_q$ is an extension of $A_q$.
    If $\abs{\Omega} < \infty$, then $\cA_q$ is the part of $\cA = \cA_2$ in $W^{-1,q}_D(\Omega)$, that is, $\dom(\cA_q) = \bigl\{u \in W^{-1,q}_D(\Omega) \colon \cA u \in W^{-1,q}_D(\Omega)\bigr\}$ and $\cA_qu = \cA u$.
  \item $\cA_q$ is sectorial with spectral angle $\phi(A_q)$, and $\cA_q+\delta$ is a positive operator on $W^{-1,q}_D(\Omega)$ for every $\delta > 0$.
    In particular, $\cA_q$ is the negative generator of a bounded holomorphic semigroup $\cT_q = E_q^{-1}T_qE_q$ on $W^{-1,q}_D(\Omega)$, where $T_q$ is the semigroup generated by $-A_q$ on $L^q(\Omega)$.
    The semigroup $\cT_q$ is an extension of $T_q$ and shares the latter's angle of holomorphy.
  \item For every $\delta>0$, the operator $\cA_q+\delta$ admits a bounded $\cH^\infty$-calculus on $W^{-1,q}_D(\Omega)$ with $\phi_\infty(\cA_q+\delta) = \phi_\infty(A_q+\delta) \leq \omega(\ft_V)$.
    If $\ft_V$ is coercive, then $\delta=0$ is also allowed.
  \end{enumerate}
\end{theorem}

\begin{proof}
  Consistency of $\cA_q$ with $A_q$ on $\dom(A_q)$ follows from $E_q$ being an extension of $(A_q+1)^{-1 /2}$ and so $A_q(A_q+1)^{-1 /2}u = (A_q+1)^{-1 /2}A_qu$ for $u \in \dom(A_q)$.
  Similarly, the resolvents of $\cA_q$ and $A_q$ are consistent.
  The part of $\cA$ in $W^{-1,q}_D(\Omega)$ is shown to be $\cA_q$ if $\abs{\Omega} < \infty$ in~\cite[Lemma~7.6]{TerElstEtAl2021} where geometry assumptions are used only to have $E_q$ as above at hand.
  The remaining statements regarding sectoriality and the semigroup are quite immediate from~\eqref{eq:w-1-q-operator} since $\cA_q$ and $A_q$ are similar.
  Consistency of the resolvents implies also consistency of the semigroups.
  Finally, see~\cite[Proposition~2.11]{DenkHieberPruess2003} for the $\cH^{\infty}$-calculus.
\end{proof}

\begin{remark}%
  \label{rem:w-1-q-square-root-coincide}
  In the context of Theorem~\ref{thm:w-1-q-results}, the inverse square root $(\cA_q+1)^{-1 /2}$ is well defined and coincides with $E_q$.
  In fact, from~\eqref{eq:w-1-q-operator}, $(\cA_q+1)^{-1 /2} = E_q^{-1}(A_q+1)^{-1 / 2}E_q$ and then it is enough to observe that by construction, $E_q^{-1}(A_q+1)^{-1 /2}$ is the identity on $L^q(\Omega)$.
\end{remark}

Let us now consider the realization of the operator $\cA_q$ on interpolation spaces between $L^q(\Omega)$ and $W^{-1,q}_D(\Omega)$.
The motivation for doing so is as follows: In many cases, if one is confronted with nonlinear evolution equations whose reaction terms do not depend only on the physical quantities themselves but also on their gradients, then ordinary negative Sobolev spaces as introduced above do not allow to adequately capture the corresponding nonlinearities.
For example, this is the case for the van Roosbroeck system~(\cite{Selberherr1984,Jerome2009}), describing the flow of electrons and holes in semiconductors, when the well-known Avalanche generation~(\cite{EbersMiller1955,HamiltonGibbonsShockley1959,Spirito1971}) is in power.
See~\cite{DisserRehberg2019,MeinlschmidtRehberg2021} for more details.
Yet, other aspects of the problem, such as inhomogeneous Neumann boundary data or surface charge densities, do not allow to revert to the $L^p(\Omega)$ scale.

It turns out that in such cases certain interpolation spaces $\bX_s \coloneqq [L^q(\Omega),W^{-1,q}_D(\Omega)]_s$ with $s \in (0,1)$ are adequate for the functional-analytic treatment of the corresponding equations~\cite{DisserRehberg2019,MeinlschmidtRehberg2021}.
Under not much stronger geometric assumptions than required so far, it is established in the seminal paper~\cite{BechtelEgert2019} that $\bX_s$ can be identified as the (anti-) dual of a Bessel potential space $H^{s,q'}_D(\Omega)$ whose elements are of fractional smoothness $s$ and incorporate a homogeneous Dirichlet boundary condition on $D$, at least so if $s > 1-\sfrac1q$.
While the precise identification of $\bX_s$ with a classical type of function space is useful in many scenarios, it is not required for the following, so we do not elaborate further.

\begin{proposition}\label{prop:bessel-scale}
  Suppose Assumption~\ref{ass:geometric-d-1-set}, let $q \in \cI(A) \cap [2,\infty)$ and let $s \in (0,1)$.
  Let $\cA_q$ be as in~\eqref{eq:w-1-q-operator} and define $\bA_q$ to be the part of $\cA_q$ in $\bX_s$.
  Then we have the following:
  \begin{enumerate}[(i)]
  \item $\bA_q$ is an extension of $A_q$ and $\cA_q$ is an extension of $\bA_q$.
  \item $\bA_{q}$ is sectorial with spectral angle $\phi(A_q)$, and $\bA_q+\delta$ is a positive operator on $\bX_s$ for every $\delta > 0$.
    In particular, $\bA_q$ is the negative generator of a bounded holomorphic semigroup $\bT_q$ on $\bX_s$ that is an extension of $T_{q}$.
    In turn, $\cT_q$ is an extension of $\bT_q$ and all semigroups share the same angle of holomorphy.
  \item For every $\delta > 0$, $\bA_q+\delta$ admits a bounded $\cH^{\infty}$-calculus on $\bX_s$ with $\phi_\infty(\bA_q+\delta) \leq \omega(\ft_V)$.
    If $\ft_V$ is coercive, then $\delta=0$ is also allowed.
  \end{enumerate}
\end{proposition}

\begin{proof}
  All properties follow directly from the definition of $\bA_q$, the fact that $\bX_s$ is a complex interpolation space between $L^q(\Omega)$ and $W^{-1,q}_D(\Omega)$, and the associated properties of $A_q$ and $\cA_q$ as in Theorem~\ref{thm:w-1-q-results}.
\end{proof}

We remark that in Proposition~\ref{prop:bessel-scale}, we see nicely the remarkable value of the square root isomorphism property~\eqref{eq:adjoint-square-root-kato}: It allowed to transfer the desirable properties of the operators from $L^q(\Omega)$ to $W^{-1,q}_D(\Omega)$, with explicit estimates due to the similarity relation~\eqref{eq:w-1-q-operator}.
By jumping over a whole order of differentiability, all the spaces $\bX_s$ \emph{in between} are covered automatically.

\section*{Comments and open questions}

We close with some concluding remarks.

\begin{enumerate}[(i)]
\item The form method used to construct the operators $A$ on $L^2(\Omega)$ does not require the particular space combination $W^{1,2}(\Omega)$ (or subspaces thereof) and $L^2(\Omega)$ but works in a very general setup in general Hilbert spaces $V$ and $H$ linked by a transference map $j \in \cL(V \to H)$ with dense range~\cite{ArendtTerElst2012}.
  It is also of interest to consider weighted spaces of the same type, that is, with a measure $\nu$ on $\Omega$ that is absolutely continuous with respect to the Lebesgue measure, as is often done in the area of Dirichlet forms.
  It is easy to extend Theorem~\ref{thm:form} to the form $\ft$ on the weighted Sobolev space $W^{1,2}(\Omega;\nu)$ in $L^2(\Omega;\nu)$, which allows to consider for example Ornstein-Uhlenbeck operators, see~\cite{ChillEtAl2006,ChillEtAl2005} and the references there.
  This then also transfers to the corresponding $L^p$-realizations.
  In the same vein, the authors are convinced that the foregoing may be generalized to forms which lead to elliptic \emph{systems}---as long as the coefficient function is (essentially) uniformly bounded and pointwise coercive.
  One can also incorporate dynamic boundary conditions by constructing the operator in hybrid $L^2$-spaces $L^2(\Omega) \oplus L^2(\Gamma)$ with a subset $\Gamma$ of the boundary $\partial\Omega$, where the measure is the sum of the Lebesgue measure on the domain and the $(d-1)$-dimensional Hausdorff measure $\cHaus_{d-1}$ on $\Gamma$, see~\cite{ChillMeinlschmidtRehberg2021,TerElstMeyriesRehberg2014,HoembergKrumbiegelRehberg2013} and the recent contribution~\cite{BoehnleinEgertRehberg2024}.
  The crucial point remains that the decisive sector inclusion is that of the \emph{form}.
\item In the context of Theorem~\ref{thm:form}, we had
  \begin{equation*}
    \omega(\ft_V) \leq \omega(\mu) \leq \arctan\Bigl(\esssup_{x \in \Omega} \frac {n(\ImOp{\mu(x)})}{m(x)}\Bigr).
  \end{equation*}
  We have already discussed sharpness of the second inequality in Remark~\ref{rem:sharpness}, yet it is also very interesting to understand whether or when the first one is sharp as well.
  In fact, there is a generic way to generate examples for which $\omega(\ft_V) < \omega(\mu)$: choose $V$ and a coefficient matrix $\mu$ such that $\mu$ is not selfadjoint, but the operator $A$ induced by the form $\ft_V$ is so.
  Such a construction seems to be hinted at in~\cite[Remark~4]{ChillEtAl2006}.
  Then $N(A)$ is real, thus so is $N(\ft_V)$ and we obtain $\omega(\ft_V) = 0$.
  But $N(\mu)$ cannot be a subset of $\R$, otherwise $\mu$ would be selfadjoint, so that $\omega(\mu) > 0$.
  Let for example $V = W^{1,2}_{\partial \Omega}(\Omega) = W^{1,2}_0(\Omega)$ with an open set $\Omega \subseteq \R^2$ and pick, say, $\mu = \left(\begin{smallmatrix} 1 & -\alpha \\ \alpha & 1 \end{smallmatrix}\right)$ with $\alpha \in \R$.
  Then $A$ becomes the Dirichlet Laplacian.
  Such examples depend on both the coefficient matrix and the subspace $V$.
  Indeed, for $V = W^{1,2}(\Omega)$ with $\Omega$ bounded with minimally smooth boundary (in any space dimension), the example does not work and one can show that $\omega(\ft_{V}) = \omega(\mu)$, see~\cite[Remark~5.12]{ArendtTerElst1997}.
  It seems however quite hard to give a precise and useful characterization of when the angles coincide or not; in one space dimension, there is always equality.
  We leave further investigations in this direction for future work.
\item The angle estimate in Theorem~\ref{thm:new-complex-angle-estimate} reproduces the angle of Theorem~\ref{thm:chill} for real coefficient functions and is optimal for symmetric ones (Remark~\ref{rem:comments-on-angle}), but it is not optimal in general: for $d = 2$ and $\mu = \left(\begin{smallmatrix}1 & -\ii \kappa\\ \ii \kappa & 1\end{smallmatrix}\right)$ with $0 < \tau_p\kappa < 1$, an elementary computation shows that the optimal angle $\theta$ with $N_p(\mu) \subseteq \Sigma_\theta$ is given by $\tan(\theta) = \sigma_p/(1-\tau_p^2\kappa^2)^{1/2}$, whereas~\eqref{eq:angle-estimate} gives $\sigma_p/(1-\tau_p\kappa)$.
  It would be interesting to find an explicit formula for the optimal angle of $N_p(\mu(x))$ in general.
\end{enumerate}

\medskip \textbf{Acknowledgment}: We wish to acknowledge, in alphabetical order, Moritz Egert (Darmstadt), Tom ter Elst (Auckland), Jochen Gl\"uck (Wuppertal), Markus Haase (Kiel), Peer Kunstmann (Karlsruhe) and Hendrik Vogt (Bremen) for stimulating discussions on the subject.

\end{document}